\documentclass{amsart}
\usepackage{amsmath,amssymb,amscd,latexsym}
\usepackage{enumerate,verbatim,calc}
\usepackage{rotate}
\usepackage[all]{xy}
\SelectTips{eu}{}
\xyoption{curve}
\CompileMatrices
\usepackage{graphics}

\usepackage[colorlinks]{hyperref}

\newcommand{\ts}{\textstyle}

\newcommand{\bwedge}{{\textstyle\bigwedge}}

\newcommand{\bph}{{\boldsymbol \varphi}}

\newcommand{\col}{\colon}

\newcommand{\hra}{\hookrightarrow}
\newcommand{\id}{\operatorname{id}}

\newcommand{\lra}{\longrightarrow}

\newcommand{\Shift}{\mathsf{\Sigma}}
\newcommand{\Spec}{\operatorname{Spec}}
\newcommand{\Supp}{\operatorname{Supp}}

\newcommand{\Hom}[3]{\operatorname{Hom}_{#1}(#2,#3)}
\newcommand{\Rhom}[3]{\operatorname{\mathsf{R}Hom}_{#1}(#2,#3)}
\newcommand{\dtensor}[1]{\otimes^{\mathsf{L}}_{#1}}

\newcommand{\dcat}[1][S]{{\mathsf D}(#1)}
\newcommand{\dcatb}[1]{{\mathsf{D}_{\mathsf f}^{\mathsf b}}(#1)}
\newcommand{\dcatc}[1]{{\mathsf{D}^{\mathsf b}_{\lift.9,\mathsf c,}}(#1)}
\newcommand{\dc}[1]{{\mathsf{D}^{}_{\lift.9,\mathsf c,}}(#1)}

\newcommand{\dcatf}[1]{{\mathsf P}(#1)}
\newcommand{\dcatg}[1]{{\mathsf G}(#1)}

\newcommand{\tra}{\twoheadrightarrow}

\newcommand{\xra}{\xrightarrow}

\newcommand{\BZ}{{\mathbb Z}}

\newcommand{\CF}{{\mathsf F}}
\newcommand{\oCF}{{\,\overline{\!\mathsf F}}}

\newcommand{\rc}[2]{\mathsf R\mathsf c^{}_{\<#1}(#2)}
\newcommand{\rp}[2]{\mathsf R\mathsf p^{}_{\<#1}(#2)}

\newcommand{\oc}[1]{\mathsf O\mathsf C(#1)}

\theoremstyle{plain}

\newtheorem{theorem}{Theorem}[section]
\newtheorem{proposition}[theorem]{Proposition}

\newtheorem{corollary}[theorem]{Corollary}

\newtheorem{itheorem}{Theorem}

\newtheorem{subtheorem}{Theorem}[subsection]
\newtheorem{subproposition}[subtheorem]{Proposition}
\newtheorem{sublemma}[subtheorem]{Lemma}
\newtheorem{subcorollary}[subtheorem]{Corollary}

\theoremstyle{definition}

\newtheorem{definition}[theorem]{Definition}

\newtheorem{example}[theorem]{Example}

\newtheorem{chunk}[theorem]{}
\newtheorem{subdefinition}[subtheorem]{Definition}
\newtheorem{subexample}[subtheorem]{Example}
\newtheorem{subremark}[subtheorem]{\rm{\emph{Remark}}}

\theoremstyle{remark}

\newtheorem{remark}[theorem]{Remark}

\font\cms=cmss10
\font\cmss=cmss8
\font\cmt=cmtex10

\newcommand{\D}{{\text {\cms\char'104}}}
\newcommand{\qc}{{\text {\cmss\char'161}{{\text {\cmss\char'143}}}}}
\newcommand{\Dqc}{\D_\qc\mkern1mu}
\newcommand\Dqcpl{\D_\qc^{\lift.95,\text{\cmt\char'053},}}
\newcommand{\Dc}{\D_{{{\text {\cmss\char'143}}}}\mkern1mu}
\newcommand\Dcmi{\D_{{{\text {\cmss\char'143}}}}^{\lift.95,\text
{\cmt\char'055},}}
\newcommand\Dcpl{\D_{{{\text {\cmss\char'143}}}}^{\lift.95,\text
{\cmt\char'053},}}

\newcommand{\Otimes}[1]{\otimes_{#1}^{\mathsf L}}
\newcommand\iso{{\mkern8mu\longrightarrow \mkern-25.5mu{}^\sim\mkern17mu}}

\newcommand{\<}{\mkern-1mu}
\renewcommand{\>}{\mkern1mu}

\newcommand{\OX}{{\mathcal O_{\<\<X}}}
\newcommand{\OY}{{\mathcal O_{\<\<Y}}}
\newcommand{\OZ}{{\mathcal O_{\<\<Z}}}

\newcommand{\OTX}{\dtensor{\<\<X}}

\newcommand{\bL}{\mathsf L}

\newcommand{\R}{\mathsf R}
\newcommand{\RH}{\mathsf{R}\>\mathcal{H}om}
\newcommand{\Rf}{\R f_{\<*}}

\newcommand{\Bigcheck}{\textup{\begin{Huge}$\mspace{-4mu}\lift{1},\check{},\mspace{4mu}$\end{Huge}}}

\newcommand{\bigcheck}[1]{{#1}^\dagger}
\newcommand{\upl}{{\lift.75,\text{\cmt\char'053},}}
\newcommand{\umi}{{\lift.85,\text{\cmt\char'055},}}

\def\lift#1,#2,{\vbox to 0pt{\vskip-#1 ex\hbox{$\scriptstyle #2$}\vss}}

\def\bilap#1{\hbox to 0pt{\hss#1\hss}}
\def\Rarrow#1{\bilap{\hbox to#1{\rightarrowfill}}}
\def\Larrow#1{\bilap{\hbox to#1{\leftarrowfill}}}
\def\Equals#1{\bilap
              {\hbox{\rule[3.5pt]{#1} {.5pt}}
               \kern-#1
               \hbox{\rule[1pt]{#1}{.5pt}}%
             }}
\begin{document}

\title[Reflexivity and rigidity over schemes]{Reflexivity and rigidity for complexes, II:\\ schemes}

\author[L.\,L.\,Avramov]{Luchezar L.~Avramov} \address{Department of Mathematics, 
University of Nebraska,  Lincoln, NE 68588, U.S.A.}  
\email {avramov@math.unl.edu}

\author[S.\,B.\,Iyengar]{Srikanth B.~Iyengar} \address{Department of Mathematics, 
University of Nebraska, Lincoln, NE 68588, U.S.A.}  
\email{iyengar@math.unl.edu}

\author[J.\,Lipman]{Joseph Lipman} \address{Department of Mathematics, 
Purdue University, W. Lafayette, IN 47907, U.S.A.} 
\email{jlipman@purdue.edu}

\thanks{Research partly supported by NSF grants DMS 0803082 (LLA) and
DMS 0903493 (SBI), and NSA grant H98230-06-1-0010 (JL).  JL thanks MSRI
for supporting a visit during which part of this paper was written.}




\begin{abstract} 
We prove basic facts about reflexivity in derived categories over
noetherian schemes; and about related notions such as semidualizing
complexes, invertible complexes, and Gorenstein-perfect maps. Also, we
study a notion of rigidity with respect to semi\-dualizing complexes,
in particular, relative dualizing complexes for Gorenstein-perfect
maps. Our results include  theorems of Yekutieli and Zhang concerning
rigid dualizing complexes on schemes. This work is a continuation of
part I, which dealt with commutative rings.
 \end{abstract}

\maketitle

\setcounter{tocdepth}{2}

\tableofcontents

\section*{Introduction} 

This paper is concerned with properties of complexes over noetherian
schemes, that play important roles in duality theory.  Some such
properties, like (derived) reflexivity, have been an integral part of
the theory since its inception; others, like rigidity, appeared only
recently.  Our main results reveal new aspects of such concepts and
establish novel links between them.  

Similar questions over commutative
rings were examined  in \cite{AIL}. Additional topics treated there are semidualizing
complexes, complexes of finite Gorenstein dimension, perfect complexes, 
invertible complexes, and rigidity with respect to semidualizing
complexes, as well as versions of these notions 
relative to  essentially-finite-type ring-homomorphisms that have finite flat 
dimension or, more generally, finite Gorenstein dimension. In this sequel we globalize such
considerations, that is, extend them to the context of schemes. 
  \medskip

This work is a substantial application of Grothendieck duality theory,
seen as the study of a twisted inverse image pseudofunctor~$(-)^!$~defined
on appropriate categories of schemes. Duality theory provides
interpretations of the local facts, a technology to globalize them,
and suggestions for further directions of development.

To place our work in context, we review two methods for proving existence
of~$(-)^!$~for noetherian schemes and separated scheme-maps of finite
type.  The original approach of Grothendieck involves the construction of
a `coherent family' of dualizing complexes; details are presented in
\cite{H} and revised in \cite{Co1}.  An alternative method, based on
Nagata compactifications and sketched in \cite{De} and \cite{Ve}, is
developed in~\cite{Lp2}.  Recent extensions of these approaches to maps
essentially of finite type provide a principal object of this study---the
concept of rigidity---and one of our main tools.

Indeed, rigid dualizing complexes over rings, introduced by Van den Bergh 
\cite{VdB} in the context of non-commutative algebraic geometry, are used by
Yekutieli and Zhang \cite{YZ1, YZ2} in an ongoing project aiming to simplify
Grothendieck's construction of~$(-)^!$, and extend it to schemes essentially 
of finite type over a regular ring of finite Krull dimension.  On the other
hand, Nayak \cite{Nk2} proved an analog of Nagata's compactification
theorem and extended the pseudofunctor~$(-)^!$ to the category of all
noetherian schemes and their separated maps essentially of finite type.
We work in this category.
  \medskip

Next we describe in some detail the notions and results of the paper.
Comparison with earlier work is postponed until the end of this
Introduction.
  \medskip

Let $X$ be a scheme, $\D(X)$ the derived category of the category of
$\OX$-modules, and $\dcatc X\subset \D(X)$ the full subcategory whose objects are the
complexes with coherent homology that vanishes in all but finitely many
degrees.  For $F$ and $A$ in $\D(X)$, we say that $F$ is \emph{derived\/
$A$-reflexive} if both $F$ and $\RH_{\<X}(F\<,A)$ are in~$\dcatc X$,
and if the canonical $\D(X)$-map is an isomorphism
$$
F\iso \RH_{\<X}\big(\RH_{\<X}(F\<,A),A\big).
$$
When $\OX$ itself is derived\/ $A$-reflexive the complex $A$ is said to be 
\emph{semidualizing}.   (The classical notion of \emph{dualizing complex} 
includes the additional requirement that $A$ be isomorphic, in $\D(X)$, 
to a bounded complex of injective sheaves.)
  \medskip

In Chapter 1 we prove basic results about semidualizing complexes in
$\D(X)$, and examine their interplay with perfect complexes, that is,
complexes $F\in\dcatc X$ 
such that for every $x\in X$ the stalk $F_x$ is 
isomorphic in $\D(\mathcal O_{\<\<X\<,\>x\>})$ to a bounded complex of 
flat $\mathcal O_{\<\<X\<,\>x\>}$-modules 
(or equivalently, such that $F$ is isomorphic in $\D(X)$ to a bounded complex
of flat $\OX$-modules).
 
  \medskip

In Chapter 2 we explore conditions on a scheme-map $f\colon X\to Y$
that allow for the transfer of properties, such as reflexivity, along
standard functors $\D(Y)\to\D(X)$.

One such condition involves the notion of \emph{perfection relative to 
$f$}, defined for $F$ in $\dcatc X$ by replacing $\mathcal O_{X,x\>}$ 
with $\mathcal O_{Y,\>f(x)\>}$ in the definition of perfection.  If this 
condition holds with $F=\OX$, then $f$ is said to be perfect.  Flat maps are classical examples.
We relate the basic global notions to ones that are local not only in
the Zariski topology, but also in the flat topology; that is, we find
that they behave rather well under faithfully flat maps. (This opens
the way to examination of more general sites, not undertaken here.)
As a sample of results concerning ascent and descent along perfect maps,
we quote from Theorem~\ref{descent} and Corollary \ref{lift semid}:

\begin{itheorem}
\label{iB-compn}
Let\/ $f\col X\to Y$ be a perfect map and $B$ a complex in\/ $\Dcpl(Y)$. 

If $M\in\D(Y)$ is derived\/ $B$-reflexive, then the complex $\bL f^*\<\<M$
in\/ $\D(X)$ is both derived\/ $\bL f^*\<\<B$-reflexive and derived\/
$f^!\<B$-reflexive. For\/ $M=\OY$ this says that if\/ $B$~is semidualizing
then so are $\bL f^*\<\<B$ and $f^!B.$

For each of these four statements, the converse holds if $M$ and\/~$B$
are in $\dcatc Y$,  and\/ $f$ is faithfully flat, or $f$ is perfect,
proper and surjective.
 \end{itheorem}

The perfection of $f$ can be recognized by its \emph{relative dualizing 
complex}, $f^!\OY$.  Indeed, $f$ is perfect if and only if $f^!\OY$ is
relatively perfect.  Furthermore, if $f$ is perfect, then
every perfect complex in $\D(X)$ is derived $f^!\OY$-reflexive.  

One sees, in particular, that when $f$ is perfect the complex $f^!\OY$
is semidualizing.  We take this condition as the definition of
\emph{G-perfect} maps. (Here G stands for Gorenstein.) They form a class  significantly
larger than that of perfect maps.  For instance, when the scheme $Y$ is
Gorenstein \emph{every} scheme map $X\to Y$ is G-perfect.  In
\S\ref{reldual}, we prove some basic properties of such maps, and, more
generally, of $\OX$-complexes that are derived $f^!\OY$-reflexive. For
such complexes there exist nice dualities with respect to the relative
dualizing complex (see Corollary~\ref{dualities}).

\emph{Quasi-Gorenstein} maps are defined by the condition that $f^!\OY$ is
perfect.  A very special case has been extensively studied: a \emph{flat}
map is quasi-Gorenstein if and only if all its fibers are Gorenstein
schemes.  On the other hand, \emph{every} map of Gorenstein schemes
is quasi-Gorenstein.  Every quasi-Gorenstein map is G-perfect.

All these classes interact in many pleasing ways with composition and
base change of scheme-maps, as explicated mainly in \S\ref{compbc}.
Such results generalize, and often strengthen, theorems about ascent and
descent along perfect maps.  For example, several assertions in Theorem
\ref{iB-compn} are obtained by taking $f=\id^X$ in the following theorem,
which is proved as part of Proposition \ref{qG and G-perfect}:

\begin{itheorem}
Let\/ $Z\xra{g}X\xra{f} Y$ be scheme-maps, with $f$ quasi-Gorenstein. 

The composition $f\<g$~is G-perfect if and only if so is~$g$.

Also, if\/ $g$ is quasi-Gorenstein then so is $f\<g$.
 \end{itheorem}

In Chapter 3 we define \emph{rigidity} with respect to an arbitrary
semi\-dualizing complex $A\in\D(X)$.  An $A$-rigid structure on $F$
in $\dcatc X$ is a $\D(X)$-isomorphism
$$
\rho\col F\iso\RH_{\<X}(\RH_{\<X}(F\<,A),F).
$$
We say that $(F\<,\rho)$ is an \emph{$A$-rigid pair}; $F\in\dcatc
X$ is an $A$-\emph{rigid complex} if such an isomorphism 
$\<\rho$ exists. Morphisms of rigid pairs are defined in the obvious 
way.   

In Theorem \ref{thm:global rigidity} we establish the basic fact about
rigid pairs:

\pagebreak[3]

\begin{itheorem}\label{icanonical}
Let $A$ be a semi\-dualizing complex in $\D(X)$.

For each quasi-coherent $\OX$-ideal $I$ such that $I^2=I,$ there exists
a canonical $A$-rigid structure on $IA\>;$ and every\/ $A$-rigid pair
is uniquely isomorphic in\/ $\D(X)$ to such an $IA$ along with its
canonical structure.
  \end{itheorem}

The theorem validates the term `rigid', as it implies that the only
automorphism of a rigid pair is the identity.  It also shows that
isomorphism classes of $A$-rigid complexes correspond bijectively to the
open-and-closed subsets of $X$. A more precise description---in terms of
those subsets---of the skeleton of the category of rigid pairs appears
in Theorem~\ref{rigid and clopen}.

In the derived category, gluing over open coverings is usually not
possible; but it is for idempotent ideals (Proposition~\ref{idem gluing}).
Consequently the uniqueness  expressed by Theorem~\ref{icanonical}
leads to gluing for rigid pairs, in the following strong sense:

\begin{itheorem}
\label{igluing}
For any open cover\/ $(U_\alpha)$ of\/ $X$ and  family\/
$(F_\alpha,\rho_\alpha)$ of\/ $A|_{U_\alpha}\!$-rigid pairs such that for
all\/ $\alpha,\;\alpha'$ the restrictions of\/ $(F_\alpha,\rho_\alpha)$
and\/ $(F_{\alpha'},\rho_{\alpha'})$ to\/ $U_\alpha\cap U_{\alpha'}$ are
isomorphic, there is a unique (up to unique isomorphism)\/ $A$-rigid pair
$(F\<,\rho)$, such that for each $\alpha$, $(F\<,\rho)|_{U_\alpha}\simeq
(F_\alpha,\rho_\alpha)$.
  \end{itheorem}

This gluing property holds even under the flat topology, see
Theorem~\ref{gluing}.\vspace{1pt}

In \S\ref{Relatively rigid complexes} we study complexes that are
\emph{relatively rigid}, that is, rigid with respect to the relative
dualizing complex $f^!\OY$ of a G-perfect map $f\colon X\to Y$
(a complex that is, by the definition of such maps, semidualizing).
As a consequence of gluing for rigid complexes under the flat topology,
gluing for relatively rigid complexes holds under the \'etale topology,
see Proposition~\ref{etale gluing}.

Relative rigidity behaves naturally with respect to (G-)perfect
maps, in the sense that certain canonical isomorphisms from duality
theory, involving relative dualizing complexes,\vspace{.5pt} respect
the additional rigid structure.  In Corollary~\ref{g! rigid} we
show that, when $g$ is perfect, the twisted inverse image functor
$g^!$ preserves relative rigidity; and also, for a composition
$Z\xrightarrow{\lift.7,g,}X\xrightarrow{\lift1.05,f,}Y$ where $f$ is
G-perfect, we demonstrate the interaction of rigidity with the
canonical isomorphism
$$
g^!\OX\Otimes{Z}\bL g^*\<\<f^!\OY\iso (f\<g)^!\OY.
$$
In Corollary~\ref{rigidity and base change} we do the same with respect to
flat base change.  Such results are obtained as applications of simple 
necessary and sufficient condition for additive functors of rigid complexes 
to be liftable to rigid pairs, detailed in Theorem~\ref{extend functors}.

The results above can be applied to complete some work started in 
\cite{AILN}.  In that paper, we associated a relative dualizing complex 
to each essentially-finite-type homomorphism of commutative rings,
but did not touch upon the functoriality properties of that complex.  
This aspect of the construction can now be supplied by using the 
fact that the sheafification of the complex in \cite{AILN} is a relative 
dualizing complex for the corresponding map of spectra; see
Example \ref{relative_for_rings}. One can then use the results in~\S\ref{Relatively 
rigid complexes}, discussed above, to enrich the 
reduction isomorphism~\mbox{\cite[4.1]{AILN}} to a functorial one. 
For such applications, it is crucial to work with scheme-maps that are \emph{essentially} of finite type; 
this is one of our reasons for choosing this category in the setup for this paper.

  \medskip

Notions and notation related to scheme-maps, as well as pertinent material
from Gro\-thendieck duality theory, as used in this paper, are surveyed
in the appendices.

  \medskip

We finish the introduction by reviewing connections to earlier work.

  \medskip

The results in Chapter 1 are, for the most part, extensions to the
global situation of results proved over commutative rings in \cite{AIL};
the transfer is fairly straightforward. 

Homomorphisms of commutative noetherian rings that track Gorenstein-type
properties were introduced and studied in \cite{AF:Gor, AF:qG, IW},
without finiteness hypotheses.  Those papers are based on Auslander and
Bridger's \cite{AB} theory of Gorenstein dimension, which is defined in
terms of resolutions by finite modules or projective modules, and so
does not globalize.  The scheme-maps defined and studied in Chapter 2
are based on a different description of finite Gorenstein dimension for
ring-homomorphisms essentially of finite type, obtained in \cite[2.2]{AIL}.

The developments in Chapter 3 are largely motivated and inspired by
work of Yekutieli and Zhang, starting with \cite{YZ}.  One of their goals 
was to construct a new foundation for Grothendieck duality 
theory.  Making extensive use of differential graded algebras (DGAs), 
in \cite{YZ1, YZ2} they extended Van den Bergh's construction~\cite{VdB} 
of rigid \emph{dualizing} complexes to schemes \emph{essentially of finite 
type over a regular ring of finite Krull dimension}, and analyzed the 
behavior of such complexes under some types of perfect maps.  Theirs 
is a novel approach, especially with regard to the introduction of DGAs 
into the subject.  However, it remains to be seen whether, once all the 
details are fully exposed, it will  prove to be simpler than the much more 
generally applicable theory presented, for example, in \cite{Lp2}.

We come to rigidity from the opposite direction, presupposing duality
theory and making no use of DGAs.  The concept obtained in this way
applies to \emph{semi}dualizing complexes over arbitrary schemes, and
behaves well under \emph{all} perfect scheme-maps.  In the setup of 
\cite{YZ2}, the regularity of the base ring implies that relative dualizing 
complexes are actually dualizing.  To compare results, one also needs 
to know that, when both apply,  our concept of rigidity coincides with 
Yekutieli and Zhang's. This follows from the Reduction Theorem 
\cite[4.1]{AILN}; see \cite[8.5.5]{AIL}.  


\section{Derived reflexivity over schemes}
  \label{globalization}
\numberwithin{equation}{subtheorem}

\emph{Rings are assumed to be commutative, and both rings and 
schemes are assumed to be noetherian.}

\subsection{Standard homomorphisms}
  \label{firstdefs}

Let $(X,\OX\<)$ be a scheme and $\D(X)$  the derived category of the
category of sheaves of $\OX$-modules.

Let $\D^\upl (X)$, resp.~$\D^\umi (X)$, be the full subcategory
of $\D(X)$ having as objects those complexes whose cohomology
vanishes in all but finitely many negative, resp.~positive,
deg\-rees;  set $\D^{\mathsf b}(X)\!:=\D^\upl (X)\cap \D^\umi
(X)$. For $\bullet = \text{\cmt\char'053}, \text{\cmt\char'055}$
or $\mathsf b$, let \mbox{$\D^\bullet_{\lift.9,\mathsf c,}(X),$
resp.~$\D^\bullet_{\lift.9,\mathsf q\mathsf c,}(X)$}, be the full
subcategory of $\D(X)$ with objects those complexes all of whose
cohomology sheaves  are  \emph{coherent,} resp.~\emph{quasi-coherent}.

To lie in $\D^{\bullet}_*(X)$ ($*=\mathsf c$ or $\mathsf q\mathsf c$,
and $\bullet = \text{\cmt\char'053}, \text{\cmt\char'055}$ or $\mathsf
b$) is a \emph{local condition}: if $(U_\alpha)$ is an open cover
of~$X\<$, then $F\in\D(X)$ lies in $\D^{\bullet}_*(X)$ if and only if
for all $\alpha$ the restriction $F|_{U_\alpha}$ lies in\vspace{1pt}
$\D^{\bullet}_*(U_\alpha)$.\looseness=-1

A number of canonical homomorphisms play a fundamental role
in this paper.

  \begin{subremark}
There is a standard trifunctorial isomorphism, relating the derived tensor 
and sheaf-homomorphism functors (see e.g., \cite[\S2.6]{Lp2}):\vspace{-1pt}
\begin{equation}
\label{sheafified adjunction}
\RH_{\<X}\<\<\big(E\Otimes{\<\<X}F,G\big)\!\iso\!\RH_{\<X}\<\big(E,\RH_{\<X}(F,G\>)\big)
\quad(E, F\<, G\in\<\D(\<X))
\end{equation}
from which one gets, by application of the composite functor $\textup{H}^0\R\Gamma(X,-)$,
\begin{equation}
\label{adjunction2}
\textup{Hom}_{\mathsf D(X)}\<\big(E\Otimes{\<\<X}F,G\big)\!\iso\!\textup{Hom}_{\mathsf D(X)}\<\big(E,\RH_{\<X}(F,G\>)\big).
\end{equation}

The map  corresponding via \eqref{adjunction2} to the identity map
of  $\>\RH_{\<X}(F,G\>)$
\begin{equation}\label{evaluation}
\varepsilon=\varepsilon^F_{\<G}\col \RH_{\<X}(F,G\>)\dtensor{\<\<X}F \to G
\qquad(F,G\in\D(X))
\end{equation}
is called   \emph{evaluation}. When $F$ is a flat complex in $\D^\umi(X)$ 
(or more generally, any q-flat complex in $\D(X)$, see \cite[\S2.5]{Lp2}), 
and $G$ is an injective complex in $\D^\upl(X)$ (or more generally, any 
q-injective complex in $\D(X)$, see \cite[\S2.3]{Lp2}),  one verifies that 
$\varepsilon$~is induced by the family of maps of complexes
$$
\varepsilon(U)\col \Hom{\OX(U)}{F(U)}{G(U)}\otimes_{\OX(U)}F(U) \to G(U)
\qquad (U\subseteq X\textup{ open})
$$ 
where,  for homogeneous
$\alpha\in\Hom{\OX(U)}{F(U)}{G(U)}$ and $b\in F(U)$,
$$
\varepsilon(U)(\alpha\otimes b)=\alpha(b).
$$ 
  \end{subremark}

Basic properties of supports of complexes are recalled for further
reference.

  \begin{subremark}
For any $F\in\D(X)$, the \emph{support of $F$} is the set
\begin{equation}
\label{defsupport} 
\Supp_{\<X}\<\<F\!:= \{\,x\in X\mid H^n(F_x)\ne 0 \textup{ for some } n \,\}.  
\end{equation} 
If $F\in\dcatc X$, then $\Supp_{\<X}\<\<F$ is a \emph{closed subset} of $X$.  
Also, for all $F$ and $G$ in~$\D^\umi_{\lift.9,\mathsf c,}(X)$, it follows 
from e.g., \cite[A.6]{AIL} that
 \begin{equation}
 \label{Supp Tensor}
\Supp_{\<X}(F\Otimes{\<\<X}\<G) =\Supp_{\<X}\<F\cap \Supp_{\<X}\<G.
\end{equation}
Note that $\Supp_{X}F=\emptyset$ if and only if $F=0$ in $\D(X)$.
  \end{subremark}

The following example opens the door to applications of the results in
\cite{AIL}.

\begin{subexample}
    \label{affine schemes}
Let $R$ be a ring. Let $\D(R)$ be the derived category of the category of 
$R$-modules, and define, as above, its full subcategories $\D^\bullet(R)$ 
for $\bullet ={}$\text{\cmt\char'053}, \text{\cmt\char'055} or $\mathsf b$. 
Let $\D^\bullet_{\mathsf f} (R)$ be the full subcategory of $\D^{\bullet}(R)$ 
having as objects those complexes whose cohomology modules are all 
\emph{finite}, i.e., finitely generated, over $R$.

For the affine scheme $X=\Spec R$, the functor that associates to each complex 
$M\in\D(R)$ its sheafification $M^\sim$ is an \emph{equivalence of categories} 
$\D^\bullet_{\mathsf f} (R)\xra{\lift.55,\approx,}\D^{\bullet}_{\lift.9,\mathsf c,}(X)$, 
\vspace{.5pt} see \cite[5.5]{BN}; when $\bullet={}$\text{\cmt\char'053} or 
$\mathsf b$, see also \cite[p.\,133, 7.19]{H}.

There is a natural bifunctorial isomorphism
\begin{equation}\label{sheafify Otimes}
(M\dtensor RN)^\sim\iso M^\sim\<\dtensor{\<\<X} N^\sim\qquad\big(M\<,\,N\in\D(R)\big);
\end{equation}
to define it one may assume that $M$ and $N$ are suitable flat complexes,
so that $\dtensor{}$ becomes ordinary $\otimes$, see \cite[\S2.5 and (2.6.5)]{Lp2}.

There is also a natural bifunctorial map  
\begin{equation}
\label{sheafify RHom}
\Rhom RMN^\sim\lra
\R\mathcal Hom_{\<X}\<\big(M^\sim\<,\>\>N^\sim\big),
\end{equation}
defined to be the one that corresponds via \eqref{adjunction2} to the composite map
$$
\smash{ 
\Rhom RMN ^\sim\<\dtensor{\<\<X}M^\sim\iso
(\Rhom RMN\dtensor RM)^\sim
\xra{\varepsilon^\sim} N^\sim,
}
$$ 
where the isomorphism comes from \eqref{sheafify Otimes}, and the 
\emph{evaluation~map} $\varepsilon$ corresponds to the identity map 
of $\Rhom RMN$ via the analog of \eqref{adjunction2} over $\D(R)$.

The map \eqref{sheafify RHom} is an \emph{isomorphism} if 
$M\in\D^\umi_{\lift.9,\mathsf f,}(R)$ and $N\in\D^\upl(R)$. To show this for 
variable $M$ and fixed $N$ one can use the ``way-out" Lemma 
\cite[p.\,68, 7.1]{H}, with~$A$ the opposite of the category of $R$-modules 
and  $P$ the family $(R^n)_{n>0}\>$, to reduce to the case $M=R$, where, 
one checks,  the map  is the obvious isomorphism.
  \end{subexample}

\subsection{Derived multiplication by global functions}
\label{Gamma acts}
Let $(X\<, \OX\<)$  be a scheme.  Here we discuss some technicalities
about the natural action of $\>\textup{H}^0(X,\OX\<)$ on~$\D(X)$.

We identify  $\textup{H}^0(X,\OX\<)$ with $\textup{Hom}_{\mathsf
D(X)}(\OX,\OX\<)$ via the ring  isomorphism that takes
$\alpha\in\textup{H}^0(X,\OX\<)$ to multiplication by $\alpha$. For
$\alpha\in\textup{H}^0(X,\OX\<)$ and $F\in\D(X)$, let $\mu_{\<F}(\alpha)$
(``multiplication by $\alpha$ in $F\>$") be the natural composite
$\D(X)$-map
$$
F\simeq
\OX\Otimes{\<\<X}F\xrightarrow{\alpha\Otimes{\<\<X}\>1}\OX\Otimes{\<\<X}F\simeq
F,
$$
or equivalently,
$$
F\simeq
F\Otimes{\<\<X}\OX\xrightarrow{1\Otimes{\<\<X}\alpha}F\Otimes{\<\<X}\OX\simeq
F.
$$

Clearly, for any $\D(X)$-map $\phi\colon F\to C$,
$$
\phi\alpha\!:=\phi\circ \mu_{\<F}(\alpha)=\mu_C(\alpha)\circ
\phi=:\alpha\phi.
$$
Furthermore, using the obvious isomorphism
$(\OX\Otimes{\<\<X}F)[1]\iso\OX\Otimes{\<\<X}F[1]$ one sees that
$\mu_F(\alpha)$ commutes with translation, that is,
$\mu_F(\alpha)[1]=\mu_{F[1]}(\alpha)$.

Thus, the family $(\mu_{\<F})_{\<F\in\mathsf D(X)\<}$ maps\vspace{.6pt}
$\textup{H}^0(X,\OX\<)$ into the ring $\mathsf C_{\<X}$ consisting
of endomorphisms of the identity functor of~$\D(X)$ that commute with
translation---the~\emph{center} of~$\D(X)$. It is straightforward to
verify that this map is an injective ring homomorphism onto the subring
of \emph{tensor-compatible} members of $\mathsf C_{\<X}$, that is, 
those~$\eta\in\mathsf C_{\<X}$ such that for all $F$, $G\in\D(X)$,
$$
\eta(F\Otimes{\<\<X} G)=\eta(F)\Otimes{\<\<X} \id^G=\id^G\Otimes{\<\<X}\>\>\eta(G).
$$

The category $\D(X)$ is $\mathsf C_{\<X}$-linear: for all $F$, $G\in\D(X)$, 
$\textup{Hom}_{\mathsf D(X)}(F,G)$ has a natural structure of 
$\mathsf C_{\<X}$-module, and composition of maps is 
$\mathsf C_{\<X}$-bilinear. So $\D(X)$ is also $\textup{H}^0(X,\OX\<)$-linear, 
via $\mu$.

\begin{sublemma}
\label{mult and RHom}
For any\/ $F,$ $G\in\D(X)$ and\/ $\D(X)$-homomorphism\/\mbox{$\alpha\colon\OX\to\OX,$} and\/ $\mu_\bullet(\alpha)$ as above, there are equalities
\[
\RH_{\<X}(\mu_{\<F}(\alpha),G) = \mu_{\RH_{\<X}\<(F,G)}(\alpha) = \RH_{\<X}(F,\mu_{\<G}(\alpha))).
\]
\end{sublemma}

\begin{proof}
Consider, for any $E\in\D(X)$, the natural trifunctorial isomorphism
$$
\tau\colon\textup{Hom}_{\mathsf D(X)}(E\Otimes{\<\<X}F\<, G)\iso
\textup{Hom}_{\mathsf D(X)}(E,\RH_{\<X}(F,G)).
$$
{From} tensor-compatibility in the image of $\mu$, and
$\textup{H}^0(X,\OX\<)$-linearity of $\D(X)$,  it follows that for any
$\alpha\in\textup{H}^0(X,\OX\<)$, the map $\mu_E(\alpha)$ induces
multiplication
by $\alpha$ in both the source and target of $\tau$. Functoriality
shows then that $\tau$ is an isomorphism
of $\textup{H}^0(X,\OX\<)$-modules.

Again, tensor-compatibility  implies that $\mu_F(\alpha)$ induces
multiplication by $\alpha$ in the source of the
$\textup{H}^0(X,\OX\<)$-linear map $\tau$, hence
also in the target. Thus, by functoriality, $\RH_{\<X}(\mu_{\<F}(\alpha),G)$
induces multiplication by $\alpha$ in the target of $\tau$. For
$E=\RH_{\<X}(F,G)$, this gives
$\RH_{\<X}(\mu_{\<F}(\alpha),G)=\mu_{\RH_{\<X}\<(F,G)}(\alpha)$. One shows similarly 
that $\RH_{\<X}(F,\mu_{\<G}(\alpha))=\mu_{\RH_{\<X}\<(F,G)}(\alpha)$.
\end{proof}

\subsection{Derived reflexivity}
\stepcounter{theorem}
\setcounter{equation}{-1}

Let $(X,\OX\<)$ be a scheme. 

One has,  for all $A$ and $F$ in~$\D(X)$, a \emph{biduality morphism}\vspace{-1pt}
\begin{equation}\label{eq:sheafbiduality}
\delta^A_F\col F\to \RH_{\<X}\big(\RH_{\<X}(F\<,A),A\big),
\end{equation}
corresponding via \eqref{adjunction2} to the natural composition\vspace{-1pt}
$$
F\Otimes{\<\<X}\RH_{\<X}(F\<,A) \iso \RH_{\<X}(F\<,A)\Otimes{\<\<X} F
\xra{\varepsilon^F_{\<A}} A.
$$

The map $\delta_F^A$ ``commutes" with restriction to open subsets
(use  \cite[2.4.5.2]{Lp2}).\vspace{1pt}

When $A$ is a q-injective complex in $\D(X)$, $\delta_F^A$ is induced by the family 
$$
\delta(U)\col F(U)\to\Hom{\OX(U)}{\Hom{\OX(U)}{F(U)}{A(U)}}{A(U)}
\qquad (U\subseteq X\textup{ open})
$$ 
of maps of complexes, where, for each $n\in F(U)$  of degree $b$, the map $\delta(U)(n)$ is 
\[
\alpha\mapsto(-1)^{ab}\alpha(n)\,,
\]
for $\alpha\in\Hom{\OX(U)}{F(U)}{A(U)}$ homogeneous of degree $a$.

\begin{subdefinition}
\label{defreflexive} Given $A$ and $F$  in $\D(X),$ 
we say  that $F$ is \emph{derived $A$-reflexive}\vspace{.6pt} if 
both $F$ and $\>\RH_{\<R}(F\<,A)$ are in~$\dcatc X$
and $\delta_F^A$ is an isomorphism.\looseness=-1 
\end{subdefinition}

This is a \emph{local condition}: for any open cover $(U_\alpha)$ of~$X\<$, $F$ is derived $A$-reflexive if and only if the same is true over every $U_\alpha$ for the restrictions of $F$ and $A$.  Also, as indicated below, if $U$ is affine, say 
$U\!:=\Spec R,$ and $C,M\in\dcatb R,$ then $M^\sim$ is\/ derived 
$C^\sim\<$-reflexive in $\D(U)\iff $ $M$ is derived $C$-reflexive in\/ $\D(R)$.

  \begin{subexample}
When $X=\Spec R$ and $M,C\in\D(R)$, it follows that with $\delta^C_M$ as in \cite[(2.0.1)]{AIL},  the map
$\delta_{M^\sim}^{C^\sim}$ factors naturally as
$$
M^\sim \xra{{(\delta^C_M\>)}^\sim} \big(\<\<\Rhom R{\Rhom RMC}C\big)^\sim
\xra{\ s\ }
 \RH_{\<X}\big(\RH_{\<X}(M^\sim\<,\>C^\sim),\>C^\sim\big)\<,
$$
where, as in \eqref{sheafify RHom}, the map $s$ is an isomorphism if 
$M\in\D^{\lift.9,\text{\cmt\char'055},}_{\mathsf f}(R)$,
$C\in\D^{\lift.9,\text{\cmt\char'053},}\<(R)$ and 
$\Rhom RMC\in\D^{\mathsf b}_{\mathsf f}(R)$.  Thus, derived 
reflexivity globalizes the notion in \cite[\S2]{AIL}.
  \end{subexample}

{From}  \cite[2.1 and 2.3]{AIL} one now gets:\vspace{-1pt}

\begin{subproposition}
\label{greflexive}
Let\/ $X$ be a noetherian scheme, and\/ $A,$ $F\in \dcatc X.$ Then
the~following conditions are equivalent.

\textup{(i)} $F$ is derived $A$-reflexive.

\textup{(ii)}  $\RH_{\<X}(F\<,\>A)\in\D^\umi(X)$ and 
there exists an isomorphism in $\D(X)$
$$
F\iso \RH_{\<X}\big(\RH_{\<X}(F\<,A),\>A).
$$

\textup{(iii)}
$\RH_{\<X}(F\<,\>A)$ is derived\/ $A$-reflexive and
$\Supp_{\<X}\<\<F\subseteq\Supp_{\<X}\<\<A$.\qed
  \end{subproposition}

\begin{subremark}
\label{dual G-perfect} 
For $A=\OX$ the theorem above shows that $F\in\dcatc X$ is derived
$\OX$-reflexive if and only if so is $\RH_{\<X}(F\<,\OX\<)$.

In the affine case, $X=\Spec R$, for any $M\in\dcatb R$, the derived
$\OX$-reflexivity of $M^{\sim}$ is equivalent to finiteness of the
Gorenstein dimension of $M\<$, as defined by Auslander and Bridger;
see \cite{AB}.
 \end{subremark}

\begin{subdefinition}
\label{defsemidual} 
An $\OX$-complex~$A$ is  \emph{semidualizing} if  $\>\OX$ is derived
$A$-reflexive. In other words, $A\in\dcatc X$ and the map $\chi^A\col \OX
\to \RH_{\<X}(A,A)$ corresponding via \eqref{adjunction2} to the natural
map $\OX\Otimes{\<\<X}A\to A\>$ is an isomorphism.
 \end{subdefinition}

As above, this condition is local on $X\<$. When $X=\Spec R$, a complex
$C\in\dcatb R$ is semidualizing  in the commutative-algebra sense (that
is, $R$ is derived $C$-reflexive, see e.g., \cite[\S3]{AIL}) if and only
if $C^\sim$ is semidualizing in the present sense.\vspace{2pt}

\begin{sublemma}
\label{semi and mult}
If $A\in\D(X)$ is semidualizing then each\/~$\D(X)$-endomorphism
of\/~$A$ is multiplication by a uniquely determined\/
$\alpha\in\textup{H}^0(X,\OX\<)$.
 \end{sublemma}

\begin{proof}
With $\chi^A\colon\OX\to\RH_{\<X}(A,A)$ as in Definition~\ref{defsemidual},
the map $\mu_A$ is easily seen to factor as follows:\vspace{-1.5pt}
\begin{align*}
\textup{Hom}_{\mathsf D(X)}(\OX,\OX\<)\xra{\!\text{via}\;\chi^A\!}
&\;\textup{Hom}_{\mathsf D(X)}\big(\OX,\RH_{\<X}(A,A)\big)\\
\cong &\;\textup{Hom}_{\mathsf D(X)}(\OX\<\<\Otimes{\<\<X}\<\<A,\>A)\\
\cong &\;\textup{Hom}_{\mathsf D(X)}(A,A).
\end{align*}
\vskip-1.5pt
\noindent The assertion results.
\end{proof}

\begin{sublemma}
\label{semid-prop}
Let \kern1pt $X$ be a noetherian scheme. \kern-1pt If\/ $A$ is a semidualizing $\OX$-complex,
then\/ $\Supp_{\<X}\<\<A=X$. Furthermore,   If 
 there is an isomorphism $A\simeq A_1\oplus A_2$ then
$\Supp_{\<X}A_1\cap \Supp_{\<X}A_2=\emptyset$. 
\end{sublemma}

\begin{proof}
The $\OX$-complex $\RH_{\<X}(A,A)$, which is isomorphic in $\D(X)$ to
$\OX$, is acyclic over the open set $X\setminus \Supp_{\<X}\<\<A$. This
implies $\Supp_{\<X}A=X$.

As to the second assertion, taking stalks at
arbitrary $x\in X$ reduces the problem to showing that if $R$ is a local
ring, and $M_1$ and $M_2$ in $\D(R)$ are such that the natural map
$$
R\to\Rhom R{M_1\oplus M_2}{\>M_1\oplus M_2}=
{\bigoplus}_{i,j=1}^2\Rhom R{M_i}{M_j}
$$
is an isomorphism, then either $M_1=0$ or $M_2=0$. 

\pagebreak[3]

But clearly, $R$
being local, at most one of the direct summands $\Rhom R{M_i}{M_j}$
can be nonzero,  so for $i=1$ or $i=2$ the identity map of $M_i$ is 0,
whence the conclusion.
 \end{proof}

\subsection{Perfect complexes}
\label{perfection}
Again, $(X,\OX\<)$ is a scheme.

  \begin{subdefinition}
    \label{defperfectcomplex}
An $\OX$-complex $P$ is \emph{perfect} if $X$ is a union of open subsets~$U$ such that the restriction $P|_U$ is $\D(U)$-isomorphic to a bounded complex of finite-rank locally free $\mathcal O_U$-modules.
  \end{subdefinition}

{From}  \cite[p.\,115, 3.5 and p.\,135, 5.8.1]{Il}, one gets:

\begin{subremark}
\label{perfect and ffd} 
The complex $P$ is perfect if and only if\/ $P\in\Dc(X)$ and $P$~is isomorphic in\/ $\D(X)$ to a bounded complex of flat\/ $\OX$-modules.
\end{subremark}

Perfection is a local condition. If $X=\Spec R$ and \mbox{$M\in\D(R)$} then 
$M^\sim$~is perfect if and only if $N$ is isomorphic in $\D(R)$ to a bounded 
complex of finite projective $R$-modules, cf.~\cite[\S4]{AIL}. 
The next result is contained in \cite[2.1.10]{Ch1}; see also \cite[4.1]{AIL}.

\begin{subtheorem}
\label{global perfect} 
$P\in\dcatc X$ is perfect if and only if so is\/ $\RH_{\<X}(P\<,\OX\<)$.\qed
\end{subtheorem} 

\begin{subproposition}
\label{perf-reflexive}
Let\/ $A$ and\/ $P$ be in $\D(X),$ with $P$ perfect. 

If\/ $F\in\D(X)$ is derived\/ $A$-reflexive then so is\/
$P\dtensor{\<\<X}F\>;$ in particular, $P$~is derived $\OX$-reflexive.
If\/ $A$~is semidualizing then $P$ is derived $A$-reflexive.
  \end{subproposition}

\begin{proof}  
The assertion being local,  we may assume that $P$ is a bounded complex
of finite\kern.5pt-rank free $\OX\<$-modules. If  two vertices of a
triangle are derived $A$-reflexive then so is the third, whence an easy
induction on the number of degrees in which $P$ is nonzero shows that
if $F$ is $A$-reflexive then so is $P\dtensor{\<\<X}F$. To show that $P$
is derived $\OX$-reflexive, take $A=\OX=F$ .

For the final assertion, take $F=\OX$.  
\end{proof}

A partial converse is given by the next result:

  \begin{subtheorem}
    \label{descent through perfect}
Let\/  $F\in\Dc(X),$  let\/ $A\in\Dcpl(X),$ and let\/ $P$ be a perfect\/ 
$\OX$-complex with\/ $\Supp_{\<X}\<\<P\supseteq \Supp_{\<X}\<\<F\<$.\vspace{1 pt} 
If\/ $P\dtensor{\<\<X}F$ is in\/ $\dcatc X,$  or $P\dtensor{\<\<X}F$ is perfect, or\/  
$P\dtensor{\<\<X}F$~is derived\/ $A$-reflexive, then the corresponding property
 holds for\/$F\<$.\looseness=-1 
  \end{subtheorem}

\begin{proof} 
The assertions are all local, and the local statements are proved in 
\cite[4.3, 4.4, and 4.5]{AIL}, modulo sheafification; see Example \ref{affine schemes}.  
\end{proof}

We'll need the following  isomorphisms (cf.~\cite[pp.\:152--153, 7.6 and 7.7]{Il}).

Let\/ $E,$ $F$ and $G$ be complexes in\/ $\D(X)$, and consider the map
\begin{equation}
\label{right}
\RH_{\<X}(E\<,F)\dtensor{\<\<X}G\to\RH_{\<X}(E\<,F\dtensor{\<\<X}G\>),
\end {equation}
corresponding via ~\eqref{adjunction2} to the natural composition
$$
\smash{
(\RH_{\<X}( E\<,\> F)\Otimes{\<\<X}G\>) \Otimes{\<\<X} E\iso
(\RH_{\<X}( E\<,\> F)\Otimes{\<\<X}E) \Otimes{\<\<X} G
\xra{\varepsilon\>\>\Otimes{\<\<X}\<1} F\Otimes{\<\<X}G.
}
$$
where $\varepsilon$ is the evaluation map from \eqref{evaluation}. 

\pagebreak[3]
\begin{sublemma} 
\label{RHom and perfect}
Let\/ $E,$ $F$ and $G$ be complexes in\/ $\D(X)$.
\begin{enumerate}[\quad\rm(1)]
\item
When either $E$ or $G$ is perfect, the map \eqref{right} is an isomorphism
\[
\RH_{\<X}(E\<,F)\dtensor{\<\<X}G\simeq\RH_{\<X}(E\<,F\dtensor{\<\<X}G\>).
\]
\item When $G$ is perfect, there is a natural isomorphism 
\[
\RH_{\<X}(E\dtensor{\<\<X}\<G,F)\simeq \RH_{\<X}(E,F)\dtensor{\<\<X}\RH_{\<X}(G,\OX\<).
\]
\end{enumerate}
\end{sublemma}

\begin{proof} (1). 
Whether the map \eqref{right} is an isomorphism is a local question, so if $E$ is perfect then one may assume that 
$E$ is a bounded complex of finite\kern.5pt-rank free $\OX$-modules. 
The affirmative answer is then given by a
simple induction on the number of degrees in which $E$
is nonzero.\looseness=-1

A similar argument applies when $G$ is perfect.\vspace{1pt}

(2). Setting $G\Bigcheck\!:=\RH_{\<X}(G,\OX\<) $, we get from (1), with $(E,F,G\>)$ changed to $(G,\OX,F)$,  an isomorphism 
\[
F\dtensor{\<\<X}G\Bigcheck\simeq G\Bigcheck\dtensor{\<\<X}F\iso\RH_{\<X}(G,F).
\]
This induces the second isomorphism below:\vspace{-5pt}
\begin{align*}
\RH_{\<X}(E\dtensor{\<\<X}\<G,F)&\iso\RH_{\<X}(E\<,\RH_{\<X}(G,F))\\[-8pt]
&\iso\RH_{\<X}(E\<,F\dtensor{\<\<X}G\Bigcheck) \\[-8pt]
&\iso \RH_{\<X}(E\<,F)\dtensor{\<\<X} G\Bigcheck;
\end{align*}
the isomorphism comes from \eqref{sheafified adjunction} and the third from (1), 
since $G\Bigcheck$ is also perfect, by Theorem~\ref{global perfect}. The desired map is 
the composite isomorphism.
\end{proof}

\subsection{Invertible complexes}
Again, $(X,\OX\<)$ is a scheme.

\begin{subdefinition}\label{definvertible}
\hskip-1pt A complex in $\D(X)$ is  \emph{invertible} if it is
semidualizing and perfect.\looseness=-1 
\end{subdefinition}

This condition is local.  If $X=\Spec R$ and $M\in\D(R)$, then $M$~is invertible in the sense of \cite[\S5]{AIL} if and only if $M^\sim$ is invertible in the present sense.

Recall that $\Shift$ denotes the usual translation (suspension) operator on complexes.

\begin{subtheorem}
\label{thm:semid}
For $L\in\dcatc X$ the following conditions are equivalent.
\begin{enumerate}[\quad\rm(i)]

\item
$L$ is invertible.\vspace{1pt}

\item
$L^{-1}\!:=\RH_{\<X}(L,\OX\<)$ is 
invertible.\vspace{1pt}

\item
Each $x\in X$ has an open neighborhood\/ $U_{\<x}$ such that for 
some integer\/ $r_{\<\<x},$ there is a $\D(U_{\<x})$-isomorphism 
$L|_{U_{\<x}}\simeq \Shift^{r_{\<\<x}}\mathcal O_{U_{\<x}}.$ 
\vspace{1pt}

\item[\rm(iii$'$)] For each connected component\/ $U$ of\/ $X$ there is an
integer\/ $r,$ a locally free rank-one $\mathcal O_U$-module\/~$\mathcal L,$ and a\/ $\D(U)$-isomorphism $L|_U\simeq \Shift^r\mathcal L$.\vspace{1pt}

\item
For some $F\in\D^{}_{\lift.9,\mathsf c,}(X)$  there is an isomorphism
$F\dtensor{\<\<X}L\simeq \OX$.
\vspace{1pt}

\item
For all $G\in\D(X)$ the evaluation map $\varepsilon$ from \eqref{evaluation}
is an isomorphism
\[
\RH_{\<X}(L,G)\dtensor {\<X} L\iso G.
\]

\item[\rm(v$'$)] 
For all $G$ and $G'$ in $\D(X),$ the natural composite map
\textup{(see \ref{sheafified adjunction})}
\begin{align*}
\RH_{\<X}({G'\dtensor{\<\<X} L},{G})\dtensor{\<\<X}L
& \!\iso\! 
\RH_{\<X}({L\dtensor{\<\<X} G'\<},{\>G})\dtensor{\<\<X}L \\
& \!\iso\!
\RH_{\<X}(L,{\RH_{\<X}({G'\<},{\>G}}))\dtensor{\<\<X}L \\
& \underset{\lift1.1,\varepsilon,}\lra \RH_{\<X}({G'\<},{\>G})
\end{align*}
\vskip-3pt
is an isomorphism.
\end{enumerate}
\end{subtheorem}

\begin{proof} 
When (i) holds, Lemma~\ref{RHom and perfect}(2), with $E=\OX$ and $G=L= F\>$, yields:
\begin{equation}
\label{LL-1}
\OX\iso\RH_X(L,L)\iso L\dtensor{\<\<X}L^{-1}.
\end{equation}

(i)${}\Leftrightarrow{}$(ii).
By Theorem~\ref{global perfect}, the $\OX$-complex $L$ is perfect
if and only if so is $L^{-1}\<$. If (i) holds, then \eqref{LL-1},
Proposition~\ref{perf-reflexive} (with $A=\OX=F$, $P=L$), and
Lemma~\ref{RHom and perfect}(1) give isomorphisms
 $$
\OX\iso L\dtensor{\<\<X}L^{-1}\iso
\RH_{\<X}(L^{-1},\OX\<)\dtensor{\<\<X}L^{-1} \iso\RH_{\<X}(L^{-1}\<,L^{-1}),
 $$
so that by Proposition~\ref{greflexive}(ii) (with $F=\OX$ and $A=L^{-1}$),
the $\OX$-module $L^{-1}$ is semidualizing; since it also perfect
(ii) holds.

The same argument with $L$ and $L^{-1}$ interchanged establishes that
(ii)${}\Rightarrow{}$(i).\vspace{1pt}

(i)${}\Rightarrow{}$(iii).
One may assume here that $X$ is affine. Then, since $L$ is invertible,
\mbox{\cite[5.1]{AIL}} gives that the stalk at $x$ of the cohomology
of~$L$ vanishes in all but one degree, where it is isomorphic to
$\mathcal O_{\<\<X\<,\>x}\>$.  The cohomology of $L$ is bounded and
coherent, therefore there is an open neighborhood $U_x$ of $x$ over
which the cohomology of $L$ vanishes in all but one degree, where it is
isomorphic to $\mathcal O_{U_x}$, i.e., (iii) holds.\vspace{1pt}

(iii)${}\Rightarrow{}$(iv).
If (iii) holds then the evaluation map~\eqref{evaluation} (with $A=L$
and $G=\OX$) is an isomorphism $L^{-1}\dtensor{\<\<X}L\iso \OX$.

(iv)${}\Rightarrow{}$(i).
This is a local statement that is established (along with some other
unstated equivalences) in \cite[5.1]{AIL}; see also \cite[4.7]{FST}.

\vspace{1pt}

(iii)${} \Rightarrow{}$(iii$'$). The function $x\mapsto r_{\<\<x}$
must be locally constant, so of constant value, say $r$, on $U$; and
then in $\D(U)$, $L\simeq \Shift^r H^{-r}(L)$. 

(iii$'$)${} \Rightarrow{}$(iii).  This implication is clear.

(i)${}\Rightarrow{}$(v). 
The first of the following isomorphisms comes from Lemma~\ref{RHom
and perfect}(2) (with $(E,F,G)=(L,G,\OX\<)$), and the second from
\eqref{LL-1}:\vspace{-2pt}
 $$
\RH_X(L,G)\dtensor{\<\<X}L \iso L^{-1}\dtensor{\<\<X}G\dtensor{\<\<X}L\iso
G.
 $$
That this composite isomorphism is $\varepsilon$ is essentially the
definition of the isomorphism
 $$
L^{-1}\dtensor{\<\<X}G=\RH_{\<X}(L,\OX\<)\dtensor{\<\<X}G\iso
\RH_{\<X}(L,G)\,;
 $$
see the proof of Lemma~\ref{RHom and perfect}.

(v)${}\Rightarrow{}$(iv). Set $F\!:=L^{-1}$, and apply  (v)
for $G=\OX$.

(v)${}\Leftrightarrow{}$(v$'$). Replace $G$ in (v) by
$\RH_{\<X}(G'\<,\>G)$; or $G'$ in~(v$'$) by $\OX$.\vspace{1pt} 
\end{proof}

\begin{subcorollary}
\label{tensor inv}
Let $L_{1}$ and $L_{2}$ be complexes in $\Dc(X)$.
\begin{enumerate}[{\quad\rm(1)}]
\item
If $L_1$ and $L_2$ are invertible, then so is $L_1\dtensor{\<\<X}L_2\>.$ 
\item
If\/ $L_1$ is in $\dcatc X$ and $L_1\dtensor{\<\<X}L_2$ is invertible, then $L_1$ is invertible.
\item
For any scheme-map $g\col Z\to X,$ if $L_{1}$ is invertible then so is $\bL g^*\< \<L_{1}$.  
\end{enumerate}
\end{subcorollary}

 \begin{proof}
For (1), use  Theorem~\ref{thm:semid}(iii$'$);  for (2),
Theorem~\ref{thm:semid}(iv)---noting that the~$F$ there may be taken to be
the invertible complex $L^{-1}$, and that tensoring with an invertible
complex takes $\Dc(X)$ into itself; and for (3), the fact that $g$
maps any connected component of $Z$ into a connected component of $X\<$.
 \end{proof}

 \begin{subcorollary}
 \label{refl and inv}
Let $A,$ $L$ and $F$ be complexes in $\dcatc X,$ with $L$ invertible.
 \begin{enumerate}[\quad\rm(1)]
  \item
$F$ is derived $A$-reflexive if and only if it is derived $L\dtensor
{\<\<X}A$-reflexive.
  \item
$F$ is derived $A$-reflexive if and only if\/ $F\dtensor{\<\<X}L$ is
derived $A$-reflexive.
  \item
$A$ is semidualizing if and only if\/ $L\dtensor{\<\<X}A$ is
semidualizing.
   \end{enumerate}
  \end{subcorollary}

  \begin{proof}
{From}, say,  Theorem~\ref{thm:semid}(iii$'$) and Lemma~\ref{RHom and
perfect}(1), one gets
  $$
\RH_{\<X}(F\<,A)\in\dcatc X\iff \RH_{\<X}(F\<,L\dtensor {\<X}A)\in\dcatc
X.
  $$
Since $L^{-1}\dtensor{\<\<X} L\simeq\OX$,  (1) follows now from 
Lemma~\ref{RHom and perfect}; (2) follows from Theorem~\ref{thm:semid}(iii);  
and (3) follows from~(1).
 \end{proof}

\begin{subremark}
\label{ex:Gorenstein}
A complex $A\in\dcatc X$ is \emph{pointwise dualizing} if every $F\in\dcatc X$ 
is derived $A$-reflexive (see \cite[6.2.2]{AIL}). Such an $A$ is semidualizing: take $F=\OX$.

It is proved in \cite[8.3.1]{AIL} that $\OX$ \emph{is pointwise dualizing} 
\emph{if and only if $X$ is a Gorenstein scheme} (i.e., the local ring 
$\mathcal O_{\<\<X\<\<,\>x}$ is  Gorenstein for all $x\in X$).

It follows from \cite[5.7]{AIL} that invertible complexes can be characterized 
as those that are semidualizing and derived $\OX$-reflexive. Hence 
\emph{when $X$ is Gorenstein, $A\in\dcatc X$ is semidualizing $\iff A$ 
is pointwise dualizing $\iff A$ is invertible.}
 \end{subremark}

\section{Gorenstein-type properties of scheme-maps}
\label{global finite G-dim}

\stepcounter{theorem}
\numberwithin{equation}{subtheorem}

\emph{All schemes are assumed to be noetherian; all scheme-maps are assumed
to be essentially of finite type \emph{(}see Appendix \emph{\ref{eft etc}}\emph{)}
and separated}.

\subsection{Perfect maps}
Let $f\col X\to Y$ be a scheme-map.

Let $f_0\col X\to Y$ denote the underlying map of topological spaces,
and $f_0^{-1}$ the left adjoint of the direct image functor~$f_{0*}$
from sheaves of abelian groups on $X$ to sheaves of abelian groups
on~$Y\<$.   There is then a standard way of making  $f_0^{-1}\OY$ into
a sheaf of commutative rings on~$X\<$, whose stalk at any point $x\in X$
is $\mathcal O_{Y,\>f(x)\>}$.

  \begin{subdefinition}
  \label{relatively perfect}
An $\mathcal{O}_X$-complex $F$ is \emph{perfect relative to}
$f$---or, as we will write, \emph{perfect over}~$f$---if it is in $\dcatc X$, and in the derived category of 
the category of $f_0^{-1}\OY$-modules $F$ is isomorphic to a 
bounded complex of flat $f_0^{-1}\OY$-modules. 

The map $f$ is \emph{perfect} if $\OX$ is perfect over $f\<$. (See \cite[p.\,250, D\'efinition 4.1]{Il}.)
  \end{subdefinition}

\begin{subremark}
\label{perfect via factorization}
Since $f$ is essentially of finite type, there is always such a $U$ for
which $f|U$ factors as (essentially smooth)$\,\circ\,$(closed immersion).
If $X\xra{i}W\to Y$ is such a factorization, then $F$ is perfect over~$f$ 
if and only if $i_*F$ is perfect over $\id^{W}\<$: the proof
of \emph{\cite[pp.\,252, 4.4]{Il}} applies here (see~Remark~\ref{diag}).

Using \cite[p.\,242, 3.3]{Il}, one~sees that perfection over $f$
is local on $X\<$, in the sense that $F$ has this property if and only
if \mbox{every $x\in X$} has an open neighborhood~$U$ such that 
$F|_U$ is perfect over $f|_U$.

Perfection over $\id^{\<X}\<\<$ is equivalent to perfection in
$\dcat[X]$; see Remark \ref{perfect and ffd}.
 \end{subremark}

Let  $\dcatf f$ be the full subcategory of $\D(X)$ whose objects are
all the complexes that are perfect over $f$; and let 
$\dcatf{X}\!:=\dcatf{\id^X\<}$ be the full subcategory of $\D(X)$ whose 
objects are all the perfect $\mathcal{O}_{X}$-complexes.

\begin{subexample}
  \label{affine perfect}
If $f\col X=\Spec S\to\Spec K=Y\<$ corresponds to a homomorphism of
noetherian rings $\sigma\col K\to S$, then $\dcatf f$ is equivalent to
the full subcategory~$\dcatf{\sigma}\subseteq\dcatb S$ with objects
those complexes $M$ that are isomorphic in $\dcat[K]$ to some
bounded complex of flat $K$-modules; this follows from \cite[p.\,168,
2.2.2.1 and p.\,242, 3.3]{Il}, in~view of the standard equivalence,
given by sheafification, between finite $S$-modules and coherent
$\mathcal{O}_{X}$-modules.\vspace{1.5pt}
 \end{subexample}

Recall that an exact functor $\mathsf{F}\col\dcat[Y]\to\dcat[X]$
is said to be \emph{bounded below} if there is an integer~$d$ such that
for all $M\in\D(Y)$ and $n\in\mathbb Z$ the following holds:
   $$
H^i(M)=0\ \textup{ for all }i<n\implies H^j(\mathsf{F}(M))= 0
\ \textup{ for all } j<n-d,
  $$
By substituting ``$>$" for ``$<$" in the preceding definition one obtains
the notion of \emph{bounded above}. If $\mathsf F$ is bounded below (resp.~bounded above) then, clearly,
$\mathsf F\>\>\D^\upl(Y)\subseteq\D^\upl(Y)$ (resp.~$\mathsf
F\>\>\D^\umi(Y)\subseteq\D^\umi(Y)$).

  \begin{subremark}
   \label{char perfect*} 
For \emph{every} scheme-map~$f\<$ the functor $\bL f^*$ is bounded
above. It is bounded below if and only if $f$ is perfect.  When $f$
is perfect, one has
 \begin{equation*}
\bL f^*\D^\mathsf{b}_\mathsf{c}(Y)\subseteq \D^\mathsf{b}_\mathsf{c}(X)\,.
  \end{equation*}
For, $\bL f^*$ is bounded above and below, hence, as above,
\mbox{$\bL f^*\D^\mathsf b(Y)\subseteq \D^\mathsf b(X)$.} Also, $\bL
f^*\dc Y \subseteq \dc X$, see \cite[p.\,99, 4.4]{H}, whose proof uses
\cite[p.\,73, 7.3]{H} and compatibility of $\bL f^*$ with open base
change to reduce to the assertion that \mbox{$\bL f^*\OY=\OX$.}
  \end{subremark}

The following characterization of perfection of $f$, in terms of the
twisted inverse image functor $f^!\<$, was proved for finite-type $f$
in \cite[4.9.4]{Lp2} and then extended to the essentially finite-type
case in \cite[5.9]{Nk2}.

  \begin{subremark}
\label{char perfect} 
For any scheme-map $f\col X\to Y$, and for all $M,N$ in $\Dqcpl(Y)$,
there is defined in \cite[\S4.9]{Lp2} and \cite[5.7--5.8]{Nk2} a
functorial $\D(X)$-map
  \begin{equation}
    \label{!tensor}
f^!M\dtensor{\<\<X}\bL f^*\<\<N\to f^!(M\dtensor{\<\<Y} N).
  \end{equation}

The following conditions on\/ $f$ are equivalent:
  \begin{enumerate}[\quad\rm(i)]
\item
The map $f$ is perfect.\vspace{1pt}
\item
The functor $f^!\col\Dqcpl(Y)\to\Dqcpl(X)$ is bounded above and below.
\item
The complex\/ $f^!\OY$ is perfect over $f$.\vspace{1pt}
\item
When $M$ is perfect, $f^!M$ is perfect over $f$; and whenever $M\dtensor YN$ 
is in $\Dqcpl(Y)$, natural the map \eqref{!tensor} is an isomorphism
  \begin{equation}
    \label{!tensoriso}
f^!M\dtensor{\<\<X}\bL f^*\<\<N\iso f^!(M\dtensor{\<\<Y}N).
  \end{equation}
    \end{enumerate}
\vskip1pt

{}From (ii) one gets, as above, $f^!{\Dqc}^{\mspace{-20mu}\mathsf
b}\mspace{8mu}(Y)\subseteq \D^\mathsf b(X)$; and the last paragraph in
\S5.4 of~\cite{Nk} gives
\begin{equation}
\label{!perfect}
f^!\Dcpl(Y)\subseteq \Dcpl(X).
\end{equation} 
Thus, for perfect $f$, one has
 \begin{equation}
f^!\D^\mathsf{b}_\mathsf{c}(Y)\subseteq \D^\mathsf{b}_\mathsf{c}(X)\>.
  \end{equation}
 \end{subremark}

\begin{subremark}
\label{PinG}
If $f\col X\to Y$ is a perfect map, then each complex that is perfect 
over~$f$ (in particular, $\OX$) is derived $f^{!}\OY$-reflexive.

This is given by~\cite[p.\,259, 4.9.2]{Il}, in whose proof ``smooth"
can be replaced by ``essentially smooth," see \cite[5.1]{AILN}.
 \end{subremark}

Next we establish some further properties of perfect maps for later use.

  \begin{sublemma}
    \label{perfect_maps}
Let  $f\col X\to Y$ be a scheme-map, and $M,B$ complexes in $\D(Y)$. 

If $f$ is an open immersion, or if $f$ is perfect, $M$ is in $\Dcmi(Y)$ and $B$ 
is in $\Dqcpl(Y),$ then~there are natural isomorphisms
\begin{align}
  \label{first_iso}
\bL f^*\RH_Y(M,B)&\iso\RH_{\<X}(\bL f^*\<\<M\<,\bL f^*\<\<B)\,,
\\
  \label{second_iso}
f^!\RH_Y(M,B)&\iso\RH_{\<X}(\bL f^*\<\<M,\>f^!\<B)\,.
\end{align}
  \end{sublemma}

  \begin{proof}
As a map in $\D(X)$, \eqref{first_iso} comes from \ref{^* and Hom}.
To show it an isomorphism we may assume $Y$ affine, say $Y=\Spec
R$. Then by \cite[5.5]{BN} and  \cite[p.\,42, 4.6.1 (dualized)]{H},
any $M\in\Dcmi(Y)$ is isomorphic to the sheafification of a complex  of
finite-rank free $R$-modules, vanishing in all sufficiently large degrees;
so \cite[p.\,181, (4.6.7)]{Lp2} gives the desired assertion.

For~\eqref{second_iso}, use \cite[4.2.3(e)]{Lp2} when $f$ is proper;
and then in the general case, compactify, see Appendix \ref{eft etc}.
\end{proof}

\begin{subremark} 
\label{Rf*perfect} 
Let $f\col X\to Y$ be a perfect and proper scheme map.

One has $\R f_*(\dcatc X)\subseteq \dcatc Y$, by \cite[p.\,237,
2.2.1]{Il}. Moreover,  if $F\in\dcatc X$ is perfect, then so is $\R
f_{*}F\>$; see Remark~\ref{perfect and ffd} and \mbox{\cite[p.\,250,
Proposition 3.7.2]{Il}}.
 \end{subremark}

  \begin{subremark}
In $\D(X)$ there is a natural map 
\[
\alpha(E,F,G)\col\RH_{\<X}(E\<,\>F)
\to \RH_{\<X}(E\Otimes{X}G\<,\>F\Otimes{X}G)
\qquad\big(E,F,G\in\D(X)\big), 
\]
corresponding via \eqref{adjunction2} to $$
\smash{(\RH_{\<X}(E\<,\>F)\Otimes{X}E)\Otimes{X}G\xra{\varepsilon\>\Otimes{\<\<X}\mathbf1}
F\Otimes{X}G} $$ where $\varepsilon$ is evaluation~\eqref{evaluation}.

Assume now that $f$ is perfect. By~Remark~\ref{char perfect} there is a natural isomorphism
\begin{equation}
\label{* and !}
\bL f^*\<\<N\dtensor{\<X} f^!\OY\simeq f^!N\qquad(N\in\Dqcpl(Y)).
\end{equation}
Hence $\alpha(\bL f^*\<\<M,\bL f^*\<\<N\<,\>f^!\OY)$ gives rise to a natural map, for all $M,N\in\Dqcpl(Y)$,
\begin{equation}
\label{an iso}
\textup{\LARGE\mathstrut}\beta(M\<,N\<,\>f)\col\<\RH_{\<X}(\bL f^*\<\<M,\bL f^*\<\<N)
\<\to\<\RH_{\<X}(f^!\<M,\>f^!\<N).
\end{equation}
  \end{subremark}

\begin{sublemma}
\label{! and *}
When $f\col X\to Y$ is perfect, $M$ is in $\dcatc Y$, and  $N$ is in $\Dqcpl(Y)$, the map $\beta(M\<,N\<,\>f)$ is an isomorphism.
\end{sublemma}

\begin{proof}
One checks, using \S\ref{any base change}, \S\ref{twisted inverse}(i), and Lemma \ref{perfect_maps},  that the question is local on both $X$ and $Y$. Hence, via \cite[p.\,133, 7.19]{H}, one may assume that $M$ is a complex of finite\kern.5pt-rank free $\OY$-modules.  

By Remarks \ref{char perfect*} and \ref{char perfect}, respectively, the functors 
$\bL f^*$ and $f^!$ are bounded both above and below.   Therefore, for every
fixed $N$ in $\Dqcpl(Y)$, the source and target of $\beta(M,N,f)$ are bounded
below functors of $M$.  Now one can argue as in the proof of \cite[p.69,  (iv)]{H}
to  reduce the problem to the case $M=\OY$.  For this $M$, one uses a similar
argument to reduce to the case where also $N=\OY$ holds.

One checks that $\beta(\OY,\OY,f)$ is the canonical map 
$\OX\!\to\<\RH_X(f^!\OY\<, f^!\OY)$, a map that, by Remark~\ref{PinG}, 
is an isomorphism.
 \end{proof}

\begin{sublemma}
\label{B-compn}
Let\/ $f\col X\to Y$ be a perfect map

When\/ $M$ is in\/ $\Dcmi(Y)$ and\/ $B$ is in\/ $\Dcpl(Y),$ the 
complex\/ $\bL f^*\<\<M$ is derived\/ $\bL f^*\<\<B$-reflexive if and 
only if it is derived\/ $f^!\<\<B$-reflexive.
 \end{sublemma}

\begin{proof}
We deal first with the boundedness conditions in
Definition~\ref{defreflexive}.  The condition  $\bL f^*\<\<M\in\dcatc X$
holds throughout, by assumption.

Assume that $\RH_{\<X}(\bL f^*\<\<M\<,\bL f^*\<\<B)$ is in $\dcatc X$. As 
$\RH_Y(M,B)\in\Dcpl(Y)$ (see \cite[p.\,92, 3.3]{H}), one gets 
from Remark~\ref{char perfect} and \eqref{first_iso} an isomorphism
\begin{equation}
\label{f^!O and otimes}
f^!\OY\<\dtensor X \RH_{\<X}(\bL f^*\<\<M,\bL f^*\<\<B)\simeq f^!\RH_Y(M,B)\,.
\end{equation}
By Remark~\ref{char perfect}(iii), $f^!\OY\in\dcatc X$, so it
follows  that $f^!\RH_Y(M,B)\in\Dcmi(X)$. On the other hand,  by
\eqref{!perfect},  $f^!\RH_Y(M,B)\in\Dcpl(X)$. We conclude that
$f^!\RH_Y(M,B)\in\dcatc X$, and  so by~ \eqref{second_iso}, that
$\RH_{\<X}(\bL f^*\<\<M,\>f^!\<B)\in\dcatc X$.

Suppose, conversely, that $\RH_{\<X}(\bL f^*\<\<M,\>f^!\<B)\in\dcatc X$, so that by \eqref{second_iso}, there is an integer~$n$ such that
$$
H^i\bigl(f^!\RH_Y(M,B))= 0\text{\ \ for all\ \  }i>n.
$$
Using \eqref{first_iso} and Remark~\ref{char perfect*} one gets:
$$
\RH_{\<X}(\bL f^*\<\<M\<,\bL f^*\<\<B)\simeq \bL f^*\RH_Y(M,B)\in\Dc(X).
$$
Also, $f^!\OY\in \dcatc X$, by Remark~\ref{char perfect}, and it follows 
from an application of (i)--(iii) in ~\ref{twisted inverse} to a local 
factorization of $f$ as (essentially smooth)$\circ$(closed immersion)---or from
Proposition~\ref{perfect reflexive}---that $\Supp_X f^!\OY=X$. So except 
for the trivial case where $X$ is empty, there is an integer $m$ such that
$$
H^mf^!\OY\ne 0\quad\text{and}\quad
H^jf^!\OY= 0\text{\ \ for all\ \  }j>m.
$$
Hence, by \eqref{f^!O and otimes}, for each $x$ in $X$ and for all all $k>n-m$,
\cite[A.4.3]{AIL} gives
$
(H^k\>\RH_{\<X}(\bL f^*\<\<M,\>\bL f^*\<\<B){)}_{\<x}=0.
$
It follows that \vspace{1pt}  $\RH_{\<X}(\bL f^*\<\<M\<,\bL f^*\<\<B)$ is in $\dcatc X$.\looseness=-1

The desired assertions now result from the isomorphisms\vspace{3pt}
\begin{align*}
\RH_{\<X}(\RH_{\<X}(\bL f^*\<\<M\<,\bL f^*\<\<B),\bL f^*\<\<B)
&\iso \RH_{\<X}(\bL f^*\RH_{\<X}(M,B),\bL f^*\<\<B)
   \\
&\iso\RH_{\<X}(f^!\RH_{\<X}(M,B),\>f^!\<B) 
   \\
&\iso \RH_{\<X}(\RH_{\<X}(\bL f^*\<\<M,f^!\<B),f^!\<B),
\end{align*}
given by formula \eqref{first_iso}, Lemma \ref{! and *}, and formula
\eqref{second_iso}, respectively.
 \end{proof}

\subsection{Ascent and descent}
\label{Ascent and descent}
Let \/ $f\col X\to Y$ be a scheme-map.

\begin{subremark}
\label{faithfully_flat}
Recall that $f$ is said to be \emph{faithfully flat} if it
is flat and surjective; and that for any flat  $f\<$, the canonical map
to $f^*$ from its left-derived functor $\bL f^*$ is an isomorphism---in
brief, $\bL f^*=f^*\<$.
\end{subremark}

\begin{sublemma}
  \label{flat_descent}
Let\/ $f\col X\to Y$ be a perfect scheme-map and\/ $M$ a complex
in\/~$\D(Y)$.

If\/ $M$ is in\/ $\dcatc Y$ then\/ $\bL f^*\<\<M$ is in\/ $\dcatc X$. The
converse holds when $M$ is in $\Dc(Y)$ and\/ $f$ is faithfully flat,
or proper and surjective.
  \end{sublemma}

  \begin{proof}
The forward implication is contained in Remark~\ref{char perfect*}.

For the converse, when $f$ is faithfully flat  there are isomorphisms
$H^n(f^*\<\<M)\cong f^*\<\<H^n(M)\ (n\in\mathbb Z)$; so it suffices
that $f^*\<\<H^n(M)=0$ imply $H^n(M)=0$.  This can be seen stalkwise,
where we need only recall, for a flat local homomorphism $R\to S$ of
local rings and any $R$-module $P\<$, that $P\otimes_R S=0$ implies $P=0$.

When $f$ is proper then by Remark~\ref{Rf*perfect}, $\R f_*(\dcatc X)\subseteq \dcatc Y$ and $\R f_{*}\OX$ is perfect. Furthermore, surjectivity of $f$ implies that
$$
\Supp_{Y}\<\R f_{*}\OX\supseteq\Supp_{Y}\<H^0\R f_{*}\OX=\Supp_{Y}\<\<f_{*}\OX=Y.
$$ 
In view of the projection isomorphism
\[
\R f_{*}\bL f^{*}\<\<M \simeq \R f_{*}\OX\Otimes{Y}M\>,
\]
see \eqref{projection2}, the desired converse follows from Theorem \ref{descent through perfect}.
\end{proof}

\begin{subproposition}
\label{perfect descent}
Let $f\col X\to Y$ be a  scheme-map and $M\in\Dcmi(Y)$.

If $M$ is perfect, then $\bL f^*\<\<M$ is perfect. The converse holds
if $f$ is faithfully flat, or if $f$ is perfect, proper and surjective.
\end{subproposition}

\begin{proof}
Suppose $M$ is perfect in $\D(Y)$. One may assume, after passing to a
suitable open cover, that $M$~is a bounded complex of finite\kern.5pt-rank
free $\OY$-modules. Then $\bL f^*\<\<M=f^*\<\<M$ is a bounded complex
of finite\kern.5pt-rank free $\mathcal O_{\<X}$-modules.  Thus if $M$
is perfect then so is $\bL f^*\<\<M$.  \vspace{1.5pt}

For the \emph{converse,} when $f$ is faithfully flat  we use the
following characterization of perfection (\cite[p.\,135, 5.8.1]{Il}):
\emph{$M\in\D(Y)$ is perfect if and only if\/ $M\in\dcatc Y$ and there are
integers\/ $m\le n$ such that for all\/ $\OY$-modules $E$ and all\/ $i$
outside the interval} $[m,n]$, $H^i(E\dtensor{\<\<Y} M\>)=0.$\looseness=-1

Writing $f^*$ in place of $\bL f^*$ (see Remark~\ref{faithfully_flat})
we have, as in the proof of
Lemma~\ref{flat_descent}, that for any $i$, the vanishing of
$$ 
H^i(f^*\<\<E\dtensor{X}\<  f^*\<\<M\>)=
H^i(f^*(E\dtensor{\<\<Y}\<\< M\>))\cong f^*H^i(E\dtensor{\<\<Y}\<\< M\>)
$$ 
implies that of $H^i(E\dtensor{\<\<Y}\<\< M\>)$. Hence the converse holds.

When $f$ is perfect, proper and surjective, one can argue as in the
last part of the proof of Lemma~\ref{flat_descent} to show that if $\bL
f^{*}\<\<M$ is perfect then $M$ is perfect.
 \end{proof}

\begin{subproposition}
\label{dir image and refl}
Let\/ $f\col X\to Y$ be a proper scheme-map and\/ $B\in\Dqcpl(Y).$

If\/ $F\in\D(X)$ is derived\/ $f^!B$-reflexive then\/ $\R\> f_*F$
is derived\/ $B$-reflexive.
\end{subproposition}

 \begin{proof}
Since $F$ and $\RH_X(F,\>f^!\<B)$ are in $\dcatc X$, it follows from
Remark~\ref{Rf*perfect} that $\R\>f_* F$ is in $\dcatc Y$, and (via
~\eqref{duality iso}) that
$$
\RH_Y(\R\>f_*F, B)\simeq\R\>f_* \RH_X(F,\>f^!\<B)\in\dcatc Y.
$$

Now apply the functor $\R\>f_*$ to the assumed isomorphism 
$$
\delta^{f^{\<!}\!B}_{\<F}\col F\iso\RH_X(\RH_X(F,f^!\<\<B),f^!\<\<B),
$$
and  use the duality isomorphism~\eqref{duality iso} twice, to get the isomorphisms
\begin{align*}
\R\>f_* F &\iso \R\>f_* \RH_X(\RH_X(F,f^!\<B),\>f^!\<B)\\
&\iso\RH_Y(\R\>f_*\RH_X(F,f^!\<B),\>B)\\
&\iso\RH_Y(\RH_Y(\R\>f_*F,B),\>B).
\end{align*}
Their composition is actually $\delta^B_{\R f_{\mkern-.5mu*}\<F}$,
though that doesn't seem so easy to show. Fortunately, owing to
Proposition~\ref{greflexive}(ii) we needn't do so to conclude that
$\R\>f_*F$ is derived $B$-reflexive.
 \end{proof}

\begin{subtheorem}
\label{descent}
Let\/ $f\col X\to Y$ be a perfect scheme-map, $M$ a complex in\/
$\Dcmi(Y)$, and $B$ a complex in\/ $\Dcpl(Y)$.

If\/ $M$ is derived $B$-reflexive, then\/ $\bL f^*\<\<M$ is derived\/ $\bL
f^*\<\<B$-reflexive and derived\/ $f^!\<\<B$-reflexive.  Conversely, if\/
$f$ is faithfully flat, or  proper and surjective,   and\/ $\bL f^*\<\<M$
is derived\- $\bL f^*\<\<B$-reflexive or derived\/ $f^!\<\<B$-reflexive,
then\/ $M$ is derived\/ $B$-reflexive.\looseness=-1
  \end{subtheorem}

\begin{proof}
Suppose first that $M$ is derived $B$-reflexive, so that, by definition,
both $M$ and $\RH_Y(M\<,B)$ are  in $\dcatc Y$. Then \eqref{first_iso}
and Remark \ref{char perfect*} show that $\bL f^*\<\<M$ and $\RH_X(
\bL f^*\<\<M\<,\bL f^*\<\<B)$ are in $\dcatc X$.   Moreover,
application of the functor~$\bL f^*$ to the $\D(Y)$-isomorphism
$M\simeq\RH_Y(\RH_Y(M\<,B),B)$ yields
 $$
\bL f^*\<\<M\simeq\RH_X(\RH_X(\bL f^*\<\<M\<,\bL f^*\<\<B),\>\bL f^*\<\<B)
 $$
in $\D(X)$. This implies that $\bL f^*\<\<M$ is derived $\bL
f^*\<\<B$-reflexive; see Proposition~\ref{greflexive}(ii).  When $B$
is in $\Dcpl(Y)$, Lemma~\ref{B-compn} yields that $\bL f^*\<\<M$ is
derived $f^!\<\<B$-reflexive.

Suppose, conversely, that  $\bL f^*\<\<M$ is derived $\bL
f^*\<\<B$-reflexive, or equivalently,  that $\bL f^*\<\<M$ is
derived $f^!\<\<B$-reflexive (see Lemma~\ref{B-compn}).  Then,
first of all, $\bL f^*\<\<M\in\dcatc X$ and, by~\eqref{first_iso},
$\bL f^*\RH_Y(M\<,B)\in\dcatc X$. Lemma \ref{flat_descent} gives then
that $M\in\dcatc Y$,  and similarly, since $\RH_Y(M\<,B)\in\Dc(Y)$
(see  \cite[p.\,92, 3.3]{H}), that $\RH_Y(M\<,B)\in\dcatc Y$.

Next, when $f$ is faithfully flat (so that  $\bL f^*=f^*\<$, see
Remark~\ref{faithfully_flat}), one checks, with moderate
effort, that if
 $$
\delta\!:=\delta^B_M\col M\to\RH_Y(\RH_Y(M\<,B),B)
 $$
\vskip-1pt\noindent is the canonical $\D(Y)$-map, then $f^*\<\delta$
is identified, via \eqref{first_iso}, with the canonical $\D(X)$-map
$\delta^{f^*\!B}_{f^*\<\<M}\>$.  The latter being an isomorphism,
therefore so are all the maps $H^n(f^*\delta)=f^*\<H^n(\delta)$. Verifying
that a sheaf-map is an isomorphism can be done stalkwise, and so,
$f$~being faithfully flat, local considerations show that the
maps~$H^n(\delta)$ are  isomorphisms. Therefore, $\delta$ is an
isomorphism.

Finally, when $f$ is proper and surjective and $\bL f^{*}\<\<M$ is
derived $f^!\<\<B$-reflexive, whence, by Proposition~\ref{dir image
and refl}, $\Rf M$ is derived $B$-reflexive,  one can argue as in the
last part of the proof of Lemma~\ref{flat_descent} to deduce that $M$
is derived $B$-reflexive.
 \end{proof}

Taking $M=\OY$ one gets:

 \begin{subcorollary}
  \label{lift semid}
Let\/ $f\col X\to Y$ be a perfect scheme-map\/ and $B\in\Dcpl(Y)$. 

If $B$ is  semidualizing, then so are\/ $\bL f^*\<\<B$ and\/ $f^!\<B$.
Conversely, if\/ $f$ is faithfully flat, or proper and surjective, and\/
$\bL f^*\<\<B$ or\/ $f^!\<B$ is semidualizing, then so is\/ $B$.\qed
 \end{subcorollary}

\begin{subcorollary} \label{descent2} Let\/ $f\col X\to Y$ be a perfect
scheme-map and\/ $M$ a complex in\/~ $\Dcmi(Y)$. Consider the following
properties: \begin{enumerate}[\quad\rm(a)]
 \item
$M$ is semidualizing.
 \item
$M$ is derived $\OY$-reflexive.
 \item
$M$ is invertible.  \end{enumerate} Each of these properties implies the
corresponding property for\/ $\bL f^*\<\<M$ in $\D(X)$.  The converse
holds when\/ $f$ is faithfully flat, or proper and surjective.
\end{subcorollary}

\begin{proof}
Note that, given Lemma~\ref{flat_descent}, we may assume that $M$ is in $\dcatc Y$.
The assertions about properties (a) and (b) are the special cases $(M,B)=(\OY,M)$ and $(M,B)=(M,\OY)$, respectively, of Theorem~\ref{descent}. The assertion about (c) follows from the assertion 
about (a) together with Proposition~\ref{perfect descent}.
\end{proof}
  
\subsection{Gorenstein-perfect maps}
  \label{reldual}
Let $f\col X\to Y$ be a scheme-map. 

\begin{subdefinition}
A \emph{relative dualizing complex} for $f$ is any $\OX$-complex 
isomorphic in~$\D(X)$ to $f^!\OY$.
\end{subdefinition}

\emph{Any relative dualizing complex is in} $\Dcpl(X)$. Indeed, 
\S\S\ref{twisted inverse}(i) and \ref{any base change} reduce the 
assertion to the case of maps between affine schemes, where the 
desired assertion follows from the following example.

\begin{subexample}
  \label{relative_for_rings}
For a homomorphism $\tau\col K\to P$ of commutative rings we write 
$\Omega_\tau$ for the $P$-module of relative differentials, and set
\[
\Omega^n_{\tau} = {\ts\bwedge}^{\<n}_P\>\>\Omega_{\tau}\quad \text{for 
each $0\le n\in\BZ$}.
\]

Let $\sigma\col K\to S$ be a homomorphism of rings that is essentially of finite type; thus, there exists a factorization
\begin{equation}
  \label{eq:factorization}
K\xra{\dot\sigma}P\xra{\sigma{\lift 1.05,',}}S
\end{equation}
where $\dot\sigma$ is \emph{essentially smooth of relative dimension} $d$ 
and $\sigma{\lift 1.3,',}$ is \emph{finite}, see~\eqref{reduce1}. As in
\mbox{\cite[(8.0.2)]{AIL},} we set 
\begin{equation}
  \label{eq:reduce2}
D^{\sigma}\!:=\Shift^d\Rhom PS{\Omega^d_{\dot\sigma}}\in\D(S).
\end{equation}

With $f\col X=\Spec S\to\Spec K=Y$ the scheme-map corresponding to~$\sigma$, 
the complex of $\OX$-modules $(D^\sigma)^\sim$ is a relative dualizing complex 
for $f$; in particular, up to isomorphism, $D^\sigma$ \emph{depends only on} 
$\sigma$, and not on the factorization~\eqref{eq:factorization}. 

Indeed, there is a $\Dqc(X)$-isomorphism
\begin{equation}
\label{eq:semidualizing, global}
f^!\OY\simeq (D^\sigma)^\sim\>;
\end{equation}
for, if $f=\dot f f'$ is the factorization corresponding to \eqref{eq:factorization} then
$$
f^!\OY\simeq {f'}^{\lift.8,!,}{\dot f}^!\OY
\simeq {f'}^{\lift.8,!,}(\Shift^d\Omega^d_{\dot\sigma})^\sim
\simeq \Shift^d\Rhom PS{\Omega^d_{\dot\sigma}}^\sim =(D^\sigma)^\sim,
$$
the second isomorphism coming from \S\ref{smooth ^!}, and
the third from \eqref{finite}.
\end{subexample}

\begin{subdefinition}
A complex $F$ in $\D(X)$ is said to be \emph{G-perfect} (for
\emph{Gorenstein-perfect}) \emph{relative to} $f$ if $F$ is 
derived $f^!\OY$-reflexive.  The full subcategory of $\dcatc X$,
whose objects are the complexes that are G-perfect
relative to~$f$ is denoted $\dcatg f$.\looseness=-1
 \end{subdefinition}

In particular, $F$ is in $\dcatg{\id^{\<X}}$ if and only if $F$ is derived 
$\OX$-reflexive. We set 
  \[
\dcatg X\!:=\dcatg{\id^X}\,.
  \]

In view of~\eqref{eq:semidualizing, global}, in the affine case G-perfection
can be expressed in terms of \emph{finite G-dimension}  in the sense
of Auslander and Bridger \cite{AB}; see \cite[\S6.3 and 8.2.1]{AIL}. 

As is the case for perfection (Remark \ref{perfect via factorization}), G-perfection 
can be tested locally.

\begin{subremark}
\label{gperfect via factorization}
A complex $F$ in $\D(X)$ is in $\dcatg f$ if and only if \mbox{every $x\in Z$} 
has an open neighborhood~$U$ such that $F|_U$ is in $\dcatg{f|_U}$.

If $f$ factors as $X\xra{i}W\xra{h} Y$  with $i$ a closed immersion and $h$ 
essentially smooth, then $F$ is in $\dcatg f$  if and only if $i_*F$ is in $\dcatg W$. 
It suffices to show this locally; and then this is  \cite[8.2.1]{AIL}, in view of the 
equivalence of categories in Example \ref{affine schemes}.
 \end{subremark}

\begin{subdefinition}
\label{defG-dim}
The map $f\col X\to Y$ is said to be \emph{G-perfect} (for \emph{Gorenstein
perfect}) if~$f^!\OY$ is semidualizing, that is,  if $\>\OX$ is in $\dcatg f$.
\end{subdefinition}

A local theory of such maps already exists:

  \begin{subexample}
\label{finite G-dim}
If $X=\Spec S$ and $Y=\Spec K\<$, where $K$ and $S$ are noetherian 
rings, and $\sigma\col K\to S$ is the ring-homomorphism corresponding 
to $f$, then $f$ is \emph{G-perfect} if and only if $\sigma$ is of 
\emph{finite G-dimension} in the sense of \cite{AF:qG}; see \cite[8.4.1]{AIL}.
  \end{subexample}

Recall from Remark \ref{char perfect} that $f$ is perfect if and only if 
$f^!\OY$ is in $\dcatf f$, the full subcategory of~$\D(X)$ whose objects 
are all the complexes that are perfect with respect to $f$.  There is a
similar description of G-perfection:

\begin{subremark}
The map $f$ is G-perfect if and only if $f^!\OY\in\dcatg f$. 
This follows from Proposition~\ref{greflexive}, since for all $x\in X\<$, 
the stalk $(f^!\OY)_x\not\simeq 0$; see~\eqref{eq:semidualizing, global}.
\end{subremark}

  \begin{subremark}
\label{Gor-target}
When $Y$ is  Gorenstein,  \emph{every} map
\mbox{$f\col X\to Y$} is G-perfect:  \cite[8.3.1]{AIL} and 
\eqref{eq:semidualizing, global} together imply that $\dcatg f=\dcatc X$.\vspace{1pt}
  \end{subremark}
  
Via \eqref{eq:semidualizing, global}, a slight generalization
of \cite[p.\,258, 4.9\emph{ff}\kern.5pt]{Il} globalizes
~\cite[1.2]{AILN}:\looseness=-1

\begin{subproposition}
\label{perfect reflexive}
Let\/ $f\col X\to Y$ be a scheme-map.

The following inclusion holds: $\dcatf f\subseteq\dcatg f$.

If\/ $M\in\dcatf Y$ then the functor\/ $\RH_{\<X}(-,\>f^!\<M)$ takes\/
$\dcatf f$ $($resp.~$\dcatg f)$ into itself; and if $M\in\dcatg Y$
then $\RH_{\<X}(-,\>f^!\<M)$ takes\/ $\dcatf f$ into\/ $\dcatg f$.
 \end{subproposition}

\begin{proof}
The first assertion is a restatement of Remark~\ref{PinG}.

The second assertion is local on $X$, so one may suppose $f$ factors as 
$X\xra{i}W\xra{h}Y$ with $i$ a closed immersion and $h$ essentially 
smooth.\vspace{.7pt} For any $F\in\dcatc X$ and $M\in\Dqcpl(Y)$ one has, 
using formula~\eqref{duality iso}, \S\ref{smooth ^!} and 
Lemma~\ref{RHom and perfect}.\looseness=-1
\[
i_*\RH_{\<X}(F, i^!h^!\<M)\simeq\RH_W(i_*F, h^!\<M)
\simeq\RH_W(i_*F, h^*\<\<M)\dtensor{\< W} h^!\OY,
\]
where $h^!\OY$ is  invertible. Consequently, by Remark~\ref{perfect via factorization}, 
\begin{align*}
\RH_{\<X}(F, f^!\<M)\in\dcatf f&\iff i_*\RH_{\<X}(F, f^!\<M)\in\dcatf W\\
&\iff \RH_W(i_*F, h^*\<\<M)\in\dcatf W.
\end{align*}

Similarly, by Remark~\ref{gperfect via factorization} and
Corollary~\ref{refl and inv}(2), 
  \begin{align*}
\RH_{\<X}(F, f^!\<M)\in\dcatg f&\iff i_*\RH_{\<X}(F, f^!\<M)\in\dcatg h\\
&\iff \RH_W(i_*F, h^*\<\<M) \textup{ is derived $\mathcal O_{\<W}$-reflexive}.
  \end{align*}

If $F\in\dcatf f$ then $i_*F$ is a perfect $\mathcal O_{\<W}$-complex, and 
by Lemma~\ref{RHom and perfect}(2), 
  \begin{equation}
  \label{RHom}
\RH_W(i_*F, h^*\<\<M)\simeq h^*\<\<M\dtensor W\RH_W(i_*F, \mathcal O_{\<W}),
  \end{equation}
where $\RH_W(i_*F, \mathcal O_W)$ is perfect (see Theorem~\ref{global
perfect}).

If $M\in\dcatf Y$ then by Proposition~\ref{perfect descent},
$h^*M\in\dcatf W$, and then \eqref{RHom} shows that $\RH_W(i_*F,
h^*\<\<M)\in\dcatf W$. Thus $\RH_{\<X}(F, f^!\<M)\in\dcatf f$.

If $M\in\dcatg Y$, then $h^*\<\<M$ is derived $\mathcal O_{\<W}$-reflexive,  hence so is $\RH_W(i_*F,
h^*\<\<M)$; see Theorem~\ref{descent}, \eqref{RHom} and Proposition~\ref{perf-reflexive}. So
$\RH_{\<X}(F, f^!\<M)\in\dcatg f$.

If $F\in\dcatg f$ and $M\in\dcatf Y$ then $i_*F\in\dcatg h$ is $\mathcal O_{\<W}$-reflexive and $h^*\<\<M$ is perfect; so by Lemma~\ref{RHom and perfect}(1), \eqref{RHom} still holds, so  $\RH_W(i_*F, h^*\<\<M)$ is $\mathcal O_{\<W}$-reflexive, by Remark~\ref{dual G-perfect} and Proposition~\ref{perf-reflexive}. So again, 
$\RH_{\<X}(F, f^!\<M)\in\dcatg f$.
\end{proof}

{From} Proposition~\ref{perfect reflexive} one gets the following
result. It can also be seen as the special case $g=\id^Y$ of
Proposition~\ref{compn} below.

\begin{subcorollary}
\label{perfect is G-perfect}
Any perfect map is G-perfect. \qed
\end{subcorollary}

Applying Proposition~\ref{perfect reflexive} to  $\RH_{\<X}(\OX,f^!F)$,
one gets: \begin{subcorollary} If\/ $f\col X\to Y$ is perfect then\/
$f^!\dcatf Y\subseteq\dcatf f$ and\/ $f^!\dcatg Y\subseteq\dcatg f$.
If\/ $f$ is G-perfect then\/ $f^!\dcatf Y\subseteq\dcatg f$.\qed
\end{subcorollary}

Also, in view of Proposition~\ref{greflexive}(iii):

  \begin{subcorollary} \label{dualities}
For any scheme-map\/ $f\col X\to Y\<,$ the\/ \emph{relative dualizing
functor} $\RH_{\<X}(-,f^!\OY)$ induces a commutative diagram of
categories, where horizontal arrows represent equivalences:\vspace{-3pt}
\[
\xymatrixcolsep{3.5pc}
\xymatrixrowsep{1.5pc}
\xymatrix {
{\dcatg{f}}^{\mathsf{op}}\, \ar@{<->}[r]^-{\equiv}\ar@{}[d]|{\bigcup\!|}
&{\,\dcatg{f}}\ar@{}[d]|{\bigcup\!|}
\\
\dcatf{f}^{\mathsf{op}}\,\ar@{<->}[r]^-{\equiv}
& \,\dcatf{f} 
}
\]
These equivalences are \emph{dualities,} in the sense of\/  \cite[\S6]{AIL}. \qed
\end{subcorollary}

\pagebreak[3]

\subsection{Quasi-Gorenstein maps}
For the following notion of quasi-Gorenstein map, cf.~\cite[2.2]{AI} and \cite[\S
8.6.1]{AIL}. For the case when $f$ is flat, see also
\cite[p.\,298, Exercise~9.7]{H}, which can be done, e.g.,  along the
lines of the proof of \cite[Lemma 1]{Lp1}.)\vspace{1pt}

\begin{subdefinition}
A map $f\col X\to Y$ is \emph{quasi-Gorenstein} if $f^!\OY$ is
invertible. If, in addition, $f$ is perfect, then\/ $f$ is said to be
a \emph{Gorenstein map}.
 \end{subdefinition}

If $f\col X\to Y$ is quasi-Gorenstein then, clearly, $\OX\in\dcatg
f$, i.e., $f$ \emph{is G-perfect.}  More generally,
Corollary~\ref{refl and inv} shows that  $\dcatg f=\dcatg X$.

\begin{subexample}
Let $f\col X\to Y$ be a scheme map.  If $X$ is Gorenstein and
$f$ is G-perfect, then $f$ is quasi-Gorenstein; see
Example~\ref{ex:Gorenstein}. Example~\ref{Gor-target} shows then
that when $X$ and $Y$ are both Gorenstein  $f$ is quasi-Gorenstein.
\end{subexample}

One has the following globalization of the \emph{flat} case of
\cite[8.6.2]{AIL}, see also \cite[2.4]{AI}:
\begin{subproposition} 
If\/  $f\col X\to Y$ is a flat Gorenstein map, with diagonal map\/
$\delta\col X\to X\times_Y X,$ then there are natural isomorphisms
$$ W_{\<\<f}\!:=\mathcal Hom_{\<X}(\delta^!(\mathcal O_{\<X\times_Y
X}),\>\>\OX\<) \underset{\lift1.05,\nu,}{\iso}\RH_{\<X}(\delta^!(\mathcal
O_{\<X\times_Y X}),\>\>\OX\<) \iso f^!\OY.  $$ If furthermore
$g\col Z\to X$ is  finite, then \eqref{duality iso} gives a natural
isomorphism
 \[
g_*(f\<g)^!\OY\cong \R\>g_*g^!\<f^!\mathcal
O_Y\iso \RH_{\<X}(g_*\mathcal O_{\<Z},W_{\<\<f})\,.
 \]
\end{subproposition}

\begin{proof}
For \emph{any} flat scheme-map $f\col X\to Y$ there is a natural isomorphism
$$
\delta^!(\mathcal O_{\<X\times_Y X})\iso
\RH_{\<X}(f^!\OY\<, \OX\<)
$$ 
(see Corollary 6.5 in \cite{AILN}, with $M=\OX=N$). 

It follows, when $f^!\OY$ is invertible, that the complex
$\delta^!(\mathcal O_{\<X\times_Y X}\mkern-.5mu)$ is invertible, and
that there is a natural $\D(X)$-isomorphism
$$
f^!\OY\iso\RH_{\<X}(\delta^!(\mathcal O_{\<X\times_Y X}),\>\>\OX\<).
$$

That the natural map $\nu$ is an isomorphism holds true with any 
perfect complex in place of $\delta^!(\mathcal O_{\<X\times_Y X})$:
the assertion is local, hence reduces to the corresponding (obvious) 
assertion for rings.

For the final assertion, note that the natural map is an isomorphism
$$
g_*(f\<g)^!\OY\iso\R\>g_*(f\<g)^!\OY
$$
because the equivalence of categories given in \cite[p.\,133, 7.19]{H}
allows us to work exclusively with quasi-coherent sheaves, on which the
functor $g_*$ is exact.  \end{proof}

\subsection{Composition, decomposition, and base change}
\label{compbc}
We turn now to the behavior of relative perfection and
G-perfection, especially vis-\`a-vis the derived direct-
and inverse-image functors and  the twisted inverse image functor,
when several maps are involved.

Generalizing Proposition \ref{perfect descent} (which is the special
case $f=\id^{\<X}$), one has:

\begin{subproposition}[cf.~{\cite[pp.\,253--254, \kern-1pt 4.5.1]{Il}}]
\label{f-perfect descent}
Let\/ $Z\xra{g}X\xra{f}Y$ be scheme-maps, with $g$ perfect. 

Then\/ $\bL g^*\dcatf f\subseteq\dcatf {f\<g}$.  
In particular, if\/ $f$ is perfect then so is $f\<g$.

Conversely, if $\>g\<$~is faithfully flat, or if\/ $g$ is proper and
surjective and $F\in\dc X$, then\/ $\bL g^*\<\<F\<\in\dcatf{f\<g}\implies
F\<\in\dcatf f$.  In particular, if\/ $f\<g$ is perfect then so is\/ $f$.
\end{subproposition}

\begin{proof} 
Let $F\in\dcatf f$. By Lemma~\ref{flat_descent}, $\bL g^*\<\<F\in\dcatc
Z$. Hence by \cite[p.\,242, 3.3, p.\,251, 4.3 and p.\,115, 3.5(b)]{Il}
(whose proofs are easily made to apply to essentially finite\kern.5pt-type
maps of noetherian schemes), for $\bL g^*\<\<F$ to be in $\dcatf{f\<g}$
it suffices that there be integers $m\le n$ such that  for any
$\OY$-module~$M$ and integer~$j\notin[m,n]$,
$$
0=H^j(\bL g^*\<\<F\dtensor{\<Z} \bL (f\<g)^*\<\<M)
\cong H^j(\bL g^*(F\dtensor{\<\<X}  \bL f^*\<\<M)).
$$
But by \emph{loc.\,cit.} this holds because $F$ is in $\dcatf g$ and
$\bL g^*$ is bounded.

Taking $M=\OY$ one gets that if $f$ is perfect then $f\<g$ is
perfect.\vspace{1pt}

For the converse, if $g$ is faithfully flat (so that $\bL g^*=g^*$) then
for any $\OX$-module~$F$ and any $j\in\mathbb Z$, one sees stalkwise  that
\[
H^j(g^*\<\<F\>)\cong g^*\<H^j(F\>)=0\iff H^j(F\>)=0.
\]
Hence if $F\in\Dc(X)$ and $g^*\<\<F\in\dcatf{f\<g}\subseteq\dcatc
Z$---whence $F\in\dcatc X$---then by an argument like that above,
$F\in\dcatf f$.

In the remaining case one argues as in the proof of Proposition
\ref{perfect descent}.  (It should be noted that the relevant part of
Theorem~\ref{descent through perfect} is proved via the above criterion
for relative perfection, so it applies not only to perfection but more
generally to relative perfection.)  \end{proof}

Analogously, for $A\!:=f^!\OY$ one has $(f\<g)^!\OY\simeq g^!\<A$, so
Theorem~\ref{descent} gives 

\begin{subproposition}[Cf.\,{\cite[4.7]{AF:qG}}]
\label{compn}
Let\/ $Z\xra{g}X\xra{f}Y$ be scheme-maps, with $g$ perfect. 

Then\/ $\bL g^*\dcatg f\subseteq\mathsf G(f\<g)$.  In particular, if $f$
is G-perfect then so is $f\<g$.\vspace{1pt}

Conversely,  if\/  $g$~is  faithfully flat and\/ $F\in\Dcmi(X)$,
or if\/ $\>g$ is proper and surjective and $F\in\Dc(X)$, then\/ $\bL
g^*\<\<F$ in $\dcatg{f\<g}$ implies $F\in \dcatg f$.
  \qed \end{subproposition}

The next proposition generalizes parts of Proposition~\ref{perfect reflexive}. 
The proof is quite similar, and so is omitted.

\begin{subproposition}
  \label{more perfect reflexive} 
Let\/ $Z\xra{g}X\xra{f}Y$ be scheme-maps, $P\in\dcatf g$, $F,A\in\D(X)$.

If\/ $F\in\dcatf f$ then $\RH_Z(P\<,\>g^!F)\in\dcatf{f\<g}$. 
\textup{(Cf.~\cite[p.\,258, 4.9]{Il}.)} 
In other words, the functor $\RH_Z(-,\>g^!\<F)$ takes\/ $\dcatf g$ to\/ 
$\dcatf{f\<g}$.

If $F$ is $A$-reflexive then $\RH_Z(P\<,\>g^!\<F)$ is $g^!\<A$-reflexive.
For\/ $A=f^!\OY$ this gives that  $\RH_X(-,\>g^!\<F)$ takes\/ $\dcatf g$ 
to\/ $\dcatg{f\<g}$.\qed
\end{subproposition}

\begin{subproposition}
\label{g^! preserves}  
Let\/ $Z\xra{g}X\xra{f}Y$ be scheme-maps, with $g$ perfect.

Then\/ $g^!\dcatf f\subseteq\dcatf{f\<g}$ and\/ $g^!\dcatg
f\subseteq\dcatg{f\<g}$.

Conversely, if\/ $g$ is proper and surjective, $F$ is in $\Dcpl(X)$, and
$g^!F$ is in $\dcatf{f\<g}$ $($resp.~$\dcatg{f\<g})$ then $F$ is in
$\dcatf{f}$ $($resp.~$\dcatg{f})$.  \end{subproposition}

\begin{proof} 
The direct assertions are obtained from Proposition~\ref{more perfect
reflexive} by taking $P=\OZ$.

If $g$ is perfect then $g^!\OX\in\dcatf g$ and 
$$
\R g_*g^!F\simeq \R g_*(g^!\OX\dtensor{\<Z} \bL g^*\<\<F)
\simeq \R g_*g^!\OX\dtensor{\<\<X} F\,;
$$
see Remark~\ref{char perfect}. If $g$ is also proper then $\R g_*g^!\OX$
is perfect \cite[p.\,257, 4.8(a)]{Il}. One can then argue as at the end
of the proof of Proposition~\ref{f-perfect descent}.  \end{proof}

\begin{subproposition}
\label{dir image}
Let\/ $Z\xra{g}X\xra{f}Y$ be scheme-maps, with $g$ proper. 

Then $\R\>g_*\dcatf{f\<g}\subseteq\dcatf f$ and 
$\R\>g_*\dcatg{f\<g}\subseteq\dcatg f$.
\end{subproposition}

\begin{proof} 
For $\mathsf P$ one can proceed as in \cite[p.\,257, 4.8]{Il}. (This
ultimately uses the projection isomorphism \eqref{projection}.)

For $\mathsf G$ apply Proposition~\ref{dir image and refl} with $B=f^!\OY$.
\end{proof}

\begin{subproposition}[Cf.~{\cite[5.2]{IW}}]
\label{qG and G-perfect} 
Let\/ $Z\xra{g}X\xra{f}Y$ be scheme-maps, with $f$ quasi-Gorenstein. 

Then $\dcatg{f\<g}=\dcatg g$.  In particular,  $f\<g$~is
G-perfect if and only if so is~$g$.

Also, if\/ $g$ is quasi-Gorenstein then so is $f\<g$.
\end{subproposition}

\begin{proof}  
For any invertible $F\in\D(X)$  the natural map (see \ref{!tensor})
$$
g^!\OX\Otimes{Z}\bL g^*\<\< F\to g^!\<F
$$
is an \emph{isomorphism}: the question being local (see~\S\ref{any
base change}), one reduces  via ~\ref{thm:semid}(iii$'$) to the simple
case $F=\OX$.

When $F$ is the invertible complex $f^!\OY\<$, there results an isomorphism
  \[
g^!\OX\Otimes{Z}\bL g^*\<\<f^!\OY\to g^!\<f^!\OY\simeq(f\<g)^!\OY.
  \]
The first assertion follows from Corollary~\ref{refl and inv}(1)
(with $A=g^!\OX$, $L=\bL g^*\<\<f^!\OY$); and the last holds because if
$g^!\OX$ is invertible then by Corollary~\ref{tensor inv}, $(f\<g)^!\OY$
is invertible as well.  \end{proof}

The last assertion of Proposition~\ref{qG and G-perfect}
expresses a composition property of quasi-Gorenstein homomorphisms. Here
is a decomposition property:

\begin{subproposition}[Cf.~{\cite[4.6]{AF:Gor}, \cite[5.5]{IW}}]
Let\/ $Z\xra{g}X\xra{f}Y$ be scheme-maps, with\/ $g$ perfect. 

If\/ $f\<g$ is quasi\kern.5pt-Gorenstein then $g$ is Gorenstein. 

Suppose $g$ is faithfully flat, or proper and surjective. \kern-1pt
If\/ $f\<g$ is quasi\kern.5pt-Gorenstein $($resp.~Gorenstein$)$  then so is $f$.
  \end{subproposition}

\begin{proof} By Remark~\ref{char perfect}, one has $g^!\OX\in\dcatc Z$ and
an isomorphism
 \[
g^!\OX\Otimes{Z}\bL g^*\<\<f^!\OY\to g^!\<f^!\OY\simeq(f\<g)^!\OY.
  \]
Also, the paragraph immediately before \S5.5 in \cite{Nk2} yields
$f^!\OY\in\Dc(X)$, whence $\bL g^*\<\<f^!\OY\in\Dc(Z)$. Now
Corollary~\ref{tensor inv}(2) gives the first assertion.  It  also shows
that $\bL g^*\<\<f^!\OY$ is invertible, whence so is $f^!\OY$ if $g$ is
faithfully flat, or proper and surjective (see Corollary~\ref{descent2}),
giving the quasi-Gorenstein part of the second assertion.  The last
assertion in Proposition~\ref{compn} now gives the Gorenstein part.
  \end{proof}

{From} Propositions~\ref{compn},~\ref{g^! preserves}  and~\ref{qG and
G-perfect} one gets:

\begin{subcorollary}
\label{G-dim and comm diag}
Let there be given  a commutative diagram
\[
\xymatrixrowsep{1pc}
\xymatrixcolsep{1pc}
\xymatrix{
X'\:\ar@{->}[rr]^-{v}
\ar@{->}[dd]_-{h}
&&
\:X
\ar@{->}[dd]^-{f}
\\
\\
Y'\:
\ar@{->}[rr]_-{u}
&&
\:Y^{\mathstrut}
}
\]
with\/ $u$ quasi-Gorenstein and\/ $v$ perfect.

Then $\bL v^*\dcatg f\subseteq \dcatg h$ and $v^!\dcatg f\subseteq \dcatg
h$. Thus, when $f$ is G-perfect so is $h$.\qed
  \end{subcorollary}

It is shown in \cite[p.\,245, 3.5.2]{Il} that relative perfection is
preserved under tor-independent base change. Here is an analog (and more)
for relative G-perfection.

\begin{subproposition}
\label{G-dim and base change}
Let there be given  a tor-independent fiber square 
\emph{(}see~\S\emph{\ref{fiber square})}
\[
\xymatrixrowsep{1pc}
\xymatrixcolsep{1pc}
\xymatrix{
X'\:\ar@{->}[rr]^-{v}
\ar@{->}[dd]_-{h}
&&
\:X
\ar@{->}[dd]^-{f}
\\
\\
Y'\:
\ar@{->}[rr]_-{u}
&&
\:Y^{\mathstrut}
}
\]

If  the map $u$ is Gorenstein, or flat, or if\/ $u$ is perfect and $f$
is proper, then\/ $\bL v^*\dcatg f\subseteq \dcatg h$. In particular,
if $f$ is G-perfect then so is $h$.

Conversely, suppose that\/ $u$ is faithfully flat, or that\/
$u$ is perfect, proper, and surjective and $f$ is proper.  If\/
$F\in\dcatc X$ and\/ $\bL v^*\<\<F\in\dcatg h$, then $F\in\dcatg f$.
\end{subproposition}

\begin{proof}
In all cases, $u$ is perfect, whence so is $v$ 
\cite[p.\,245, 3.5.2]{Il}.

If $u$ is Gorenstein, the assertion is contained in Corollary~\ref{G-dim
and comm diag}.

By Lemma~\ref{flat_descent}, if $F$ is $f^!\OY$-reflexive then $\bL
v^*\<\<F$ is $\bL v^*\<\<f^!\OY$-reflexive.

If $u$ (hence $v$) is flat then by~\S\ref{any base change}, one has 
\begin{equation}\label{torindt basechange}
\bL v^*\<\<f^!\OY\cong h^!\bL u^*\OY=h^! \mathcal O_{Y'}.
 \end{equation}
Thus $v^*\<\<F$ is $h^!\mathcal O_{Y'}$-reflexive, i.e.,
$v^*\<\<F\in\dcatg h$.

The  case when $u$ is perfect and $f$ is proper is treated similarly
through the tor-independent base-change theorem \cite[4.4.3]{Lp2}.

For the converse, the assumption is, in view of the
isomorphism~\eqref{torindt basechange}, that $\bL v^*\<\<F$ is derived
$\bL v^*\<\<f^!\OY$-reflexive. Formula \eqref{!perfect} gives  that
$f^!\OY\in\Dcpl(X)$. So since $v$ satisfies all the same hypotheses as
$u$ does, Theorem~\ref{descent} yields that $F$ is $f^!\OY$-reflexive,
as asserted.  \end{proof}

\begin{subproposition}\label{Gor and base change}
Let there be given  a tor-independent fiber square
$(\<$see~\textup{\ref{fiber square})}
\[
\xymatrixrowsep{1pc}
\xymatrixcolsep{1pc}
\xymatrix{
X'\:\ar@{->}[rr]^-{v}
\ar@{->}[dd]_-{h}
&&
\:X
\ar@{->}[dd]^-{f}
\\
\\
Y'\:
\ar@{->}[rr]_-{u}
&&
\:Y^{\mathstrut}
}
\]
with either\/ $u$ flat, or\/ $u$  perfect and\/ $f$ proper. 

If the map\/ $f$ is quasi-Gorenstein $($resp.~Gorenstein$)$ then so
is\/ $h\>$.

The converse holds if\/ $u$ $($hence $v)$ is faithfully flat, or if \/
$u$ $($hence $v)$ is perfect, proper and surjective and\/ $f$ is proper.
\end{subproposition}

\begin{proof}  
As in the proof of Proposition~\ref{G-dim and base change}, one
has the isomorphism \eqref{torindt basechange}. Hence if $f^!\OY$ is
invertible then so is $h^!\mathcal O_{Y'}$ (see  Corollary~\ref{tensor
inv}(3)), whence the first quasi-Gorenstein assertion, whose converse
follows from Corollary~\ref{descent2}(c).  Also, by  \cite[p.\,245,
3.5.2]{Il}, if $f$ is perfect then so is $h$, whence the first Gorenstein
assertion, whose converse follows from the preceding converse and
Proposition~\ref{f-perfect descent} (since $u$ perfect and $h$ perfect
implies $hu=fv$ perfect).  \end{proof}

\section{Rigidity over schemes}
  \label{sec:Rigidity}

\emph{As in previous sections, schemes are assumed to be noetherian,
and scheme-maps to be essentially of finite type, and separated}.

\setcounter{equation}{0}
\numberwithin{equation}{subtheorem}

\subsection{Rigid complexes}
\label{Rigid global complexes}
Fix a scheme $X$ and a semidualizing $\OX$-complex $A$, and for
any $F\in\D(X)$ set 
$$
\bigcheck F:=\R\mathcal Hom_X(F\<, A).
$$

\begin{subdefinition}
\label{map of rigids} 
An $A$-\emph{rigid pair} $(F\<,\rho)$ is one where $F\in\dcatc X$ and 
$\rho$~is a $\D(X)$-isomorphism\vspace{-2pt}
$$
\rho\col F\iso \R\mathcal Hom_X(\>\>\bigcheck F\<, F).
$$
An $\OX$-complex $F$ is $A$-\emph{rigid} if there exists a $\rho$
such that $(F\<,\rho)$ is an $A$-rigid pair. Such a $\rho$ is called an
\emph{$A$-rigidifying isomorphism for} $F\<$.\vspace{1pt}

A \emph{morphism of\/ $A$-rigid pairs} $(F\<,\rho)\to (G,\sigma)$
is\vspace{.8pt} a $\D(X)$-map $\phi\col F\to G$ such~that the following
diagram,\vspace{.8pt} with $\tilde\phi\col \R\mathcal Hom_X(\bigcheck
F\<, F)\to \R\mathcal Hom_X(\>\bigcheck G\<, G)$ the map induced by
$\phi$, commutes:
$$
\CD
F @>\rho>>  \R\mathcal Hom_X(\bigcheck F\<, F)\\
@V\phi VV  @VV\tilde\phi V \\
G @>>\sigma>  \R\mathcal Hom_X(\>\bigcheck G\<, G)
\endCD
$$
\end{subdefinition}

The terminology ``rigid" is motivated by the fact, contained in
Theorem~\ref{unique iso}, that \emph{the only automorphism of an
$A$-rigid pair is the identity.}

  \begin{subexample}
If $R$ is a ring, $X=\Spec R$, and $M, C\in\dcatb R$ are such that $\Rhom RMC\<
\in\<\dcatb R$, then by Example~\eqref{affine schemes}, $M$ is $C$-rigid in~the
sense \mbox{of \cite[\S7]{AIL}} if and only if $M^\sim$ is $C^\sim$-rigid
in the present sense.
\end{subexample}

Since $\RH$ commutes with restriction to open subsets, an $A$-rigid pair
restricts over any open  $U\subseteq X$ to an $A|_U$-rigid pair. However,
rigidity is \emph{not} a local condition: any invertible sheaf $F$
is $F$-rigid, but $\OX$ is not $F$-rigid unless $F\cong\OX$.

On the other hand, \emph{rigid pairs glue,} in the sense explained
in Theorem~\ref{igluing} of the Introduction, and generalized in
Theorem~\ref{gluing} below.

The central result of this section, Theorem~\ref{thm:global rigidity}, a
globalization of~\cite[7.2]{AIL}, is that any $A$-rigid $F$ is isomorphic
in $\D(X)$ to $\>i_*i^*\<\<A$, with $i$ the inclusion into~$X$ of~some
open-and-closed subscheme---necessarily the support of $F\<$, see
\eqref{defsupport}; or equivalently, $F\simeq IA$ for some idempotent
$\OX$-ideal~$I$, uniquely determined by~$F$ (see Appendix~\ref{idempotent
ideals}); or equivalently, $F$ is, in $\D(X)$, a direct summand of~$A$.

\begin{subexample}
\label{gcanonical} 
The pair $(A, \rho^A)$ with\/ $\rho^A$ the natural composite isomorphism
$$
\rho^A\col A\iso \RH_{\<X}(\OX,A)
\iso \R \mathcal Hom_X(\RH_{\<X}(A,A), A),
$$
is\/ $A$-rigid.
\end{subexample}

Extending this example a little leads to:

\begin{subexample}
\label{open-and-closed}
Let $U\subseteq X$ be an open-and-closed subset, and $i\col U\hookrightarrow X$ the inclusion. Recall that the $\mathcal O_{U}$-module $i^*\<\<A$ is semidualizing; see Corollary~\ref{lift semid}. 
If $F\in\D(U)$ is $i^*\<\<A$-rigid then $i_{*}F$ is  $A$-rigid.

Indeed, if $\rho$ is an $i^*\<\<A$-rigidifying isomorphism for $F\<$, 
then one has isomorphisms
\begin{align*}
i_{*}F \underset{\lift1.3,\!i_*\rho,}\iso
&i_*\RH_U\<\big(\RH_U( F\<,i^*\<\<A),  F\big)\\[-1pt]
\iso
&i_*\RH_U\<\big(i^*\RH_{\<X}(i_{*}F,A), F\big)\\[1pt]
\iso
&\RH_{\<X}\<\big(\RH_{\<X}(i_{*}F,A), i_*F\big),
\end{align*}
where the second  comes from \eqref{^* and Hom} (since $i^*i_*F=F$), and the third
is a special case of \cite[p.\,98, (3.2.3.2)]{Lp2} (or see \cite[\S3.5.4]{Lp2}, or just reason directly, using that $i_*F$ vanishes outside $U$).

The composition of these isomorphisms is  $A$-rigidifying  for $i_*F$.
\end{subexample}

\begin{subdefinition}
\label{U-canonical}
The \emph{$U\<\<$-canonical $A$-rigid pair} $(i_*i^*\<\<A, \rho^{\>i_*i^*\!A})$ is the one constructed in Example~\ref{open-and-closed} out of the $i^*\<\<A$-rigid pair $(i^*\<\<A, \rho^{\>i^*\!A})$ in Example~\ref{gcanonical}.
\end{subdefinition}

It is well known that \emph{any} monomorphism (resp.~ epimorphism) in $\D(X)$ is split, i.e.,  has a left (resp.~ right) inverse (see e.g., \cite[1.4.2.1]{Lp2}). Thus, when we speak  of mono\kern.5pt- or epimorphisms, the adjective ``split" will usually be omitted.

\begin{sublemma}
\label{unique rho} 
Let\/ $\theta\col F\hra A$ be a  monomorphism in $\D(X)$. Let $(A,\rho^A)$ be the canonical\/ $A$-rigid pair in Example~\emph{\ref{gcanonical}}. There exists a unique\/ $A$-rigidifying isomorphism~ 
$\rho$ for\/ $F$ such that\/ $\theta$ is a morphism of rigid pairs\/ $(F\<,\rho)\to(A,\rho^A)$.
\end{sublemma}

\begin{proof}
It suffices to deal with the situation separately over each connected component of $X$; so we may assume that $X$ is connected. Then,  by Lemma~\ref{semid-prop}, either $F=0$ or $\theta$ is an isomorphism. In either case the assertion is obvious.
\end{proof}

\begin{subtheorem}
\label{thm:global rigidity}  
For any\/ $F\in\D(X)$, the following conditions are equivalent.
\begin{enumerate}[{\rm(i)}]
\item
$F$ is $A$-rigid.
\item
In\/ $\D(X),$ $F\simeq I\dtensor{\<\<X}A\simeq IA\simeq \bigcheck I$ for some idempotent\/ $\OX$-ideal\/~$I$.
\item 
There is an open-and-closed\/ $U\subseteq X$ such that, $i\col U\hookrightarrow X$ being the inclusion,
$F\simeq i_*i^*\<\<A$ in\/ $\D(X),$ whence\/ $U=\Supp_{\<X}\<\<F$.\vspace{1pt}

\item
There is, in\/ $\D(X),$ a monomorphism $F\to A$.
\end{enumerate}
When they hold, there is a unique ideal $I$ satisfying condition \emph{(ii)}.
\end{subtheorem}

\begin{proof}
(iii)${}\Rightarrow{}$(i). In view of Example~\ref{gcanonical}, this is contained in Example~\ref{open-and-closed}.

(i)${}\Rightarrow{}$(iii).
For the last assertion in (iii), since  $i_*i^*\<\<A$ vanishes outside
$U\<$,   and since for all $x\in U$ one has, in $\D(\mathcal O_{U\<\<,\>x})$,
$$
0\ne \mathcal O_{\<U\<\<,\>x}\simeq(\RH_{U}(i^*\<\<A,i^*\<\<A))_x \simeq 
\RH_{\mathcal O_{\<U\<\<,x}}((i_*i^*\<\<A)_x,(i_*i^*\<\<A)_x)
$$
therefore $U=\Supp_{\<X}( i_*i^*\<\<A\<)$.\vspace{1pt}

Now let $F$ be $A$-rigid. Then $U\!:=\Supp_{\<X}\<\<F$  \emph{is an open-and-closed subset of}~$X\<$. For, $X$~is covered by open subsets of the form $V=\Spec R$; and with $j\col V\hookrightarrow X$ the inclusion, the $j^*\<\<A$-rigid complex $j^*\<F$ (resp.~its homology) is the sheafification of $F_V\!:=\R\Gamma(V\<,F)$ (resp.~its homology), so $(\Supp_{X} F)\cap V = \Supp_R F_V$. But $F_V$ is $\R\Gamma(V\<,A)$-rigid (since $(F_V)^\sim\cong j^*\<F$ is $j^*\<\<A$-rigid), so by~\cite[7.2]{AIL}, $\Supp_R F_V=U\cap V$ is an open-and-closed subset of $V\<$. That $U$ is open-and-closed follows.

Hence, the natural map $F\to i_*i^*\<F$ is a $\D(X)$-isomorphism; so to prove the theo\-rem we can replace~$(X,A,F)$ by $(U,  i^*\<\<A, i^*\<F)$, i.e., \emph{we may assume\/ $\Supp_{\<X}\<\<F=X$.}\looseness=-1

In $\D(X)$, the complex $L\!:=\bigcheck F$ is isomorphic  to $H^0L$, which is an invertible sheaf: this assertion need only be checked locally, i.e., for affine $X\<$, where it is given by \cite[4.9]{AIL}.
(The assumptions of that theorem are satisfied because $F$ and $A$ are both in~$\dcatc X$.) The  invertible complex $L$ is derived $A$-reflexive (take $F=\OX$ in ~\ref{refl and inv}(2)); similarly, so is $L\OTX L$.  Since $\Supp_{\<X}A=X$, by Lemma~\ref{semid-prop}, therefore Proposition~\ref{greflexive}(iii) yields that  $F$ is derived $A$-reflexive. So $\bigcheck L\simeq F\<$, and
 $$
\bigcheck L\simeq \RH_{\<X}(L, \bigcheck L)\simeq
(L\dtensor{\<\<X}\<\< L)^{\lift1.1,\<\dagger,}
\qquad\textup{(see \eqref{sheafified adjunction}).}
 $$
Applying the functor $\bigcheck{}$ to these isomorphisms we get
$L\OTX L\simeq L$. Tensoring with~$L^{-1}$ shows then that $L\simeq\OX$. Thus 
$F\simeq \bigcheck L\simeq A$.

(iii)${}\Rightarrow{}$(ii). 
Associated to any open-and closed $U\subseteq X$ is the unique idempotent $\OX$-ideal $I$ that is isomorphic to $i_*\mathcal O_U$ (Corollary~\ref{idem and clopen}). For this $I$ we have natural isomorphisms, the
second from \eqref{projection} and the last two from Corollary~\ref{idem and Hom}:
$$
i_*i^*\<\<A\simeq i_*(\mathcal O_U\dtensor U i^*\<\<A)\simeq 
i_*\mathcal O_U\dtensor{\<\<X} A
\simeq I\dtensor{\<\<X} A\simeq IA\simeq \bigcheck I.
$$

(ii)${}\Rightarrow{}$(iii). Given $I$ as in (ii), let $U=\Supp_{\<X}I$, with inclusion
$i\col U\hookrightarrow X$, and use the preceding isomorphisms (see Corollary~\ref{idem and clopen}).

(iii)${}\Rightarrow{}$(iv). If $i$ is as in (iii), then
$i_*i^*\<\<A$ is a direct summand of~$A$.

(iv)${}\Rightarrow{}$(i). See Lemma~\ref{unique rho}.

It remains to note that the uniqueness of $I$ in (ii) results from 
$$
\Supp_{\<X}IA=\Supp_{\<X}(I\dtensor{\<\<X} A)=\Supp_{\<X}I\cap \Supp_{\<X}A
=\Supp_{\<X}I\cap X=\Supp_{\<X}I,
$$
see \eqref{Supp Tensor}.  The proof of Theorem~\ref{thm:global rigidity} is now completed.
\end{proof}

Define a \emph{direct decomposition} of~$F\in\D(X)$ to be a $\D(X)$-isomorphism
 \begin{equation}
 \label{orth}
F\simeq F_1\oplus F_2\oplus\cdots\oplus F_n
 \end{equation}
such that no $F_i$ vanishes; call $F$ \emph{indecomposable} if $F\ne 0$ and in any direct decomposition 
of $F\<$, one has $n=1$. Say that \eqref{orth} is an \emph{orthogonal decomposition} of~$F$ if, in addition, $F_i\Otimes{\<\<X}\<F_j=0$ for all $i\ne j$.

\begin{subcorollary}
\label{rigid decomp}
Let\/ $F\ne 0$ be an $A$-rigid complex.\vspace{1pt} Let\/ $\Supp_{\<X}\<\<F =\bigsqcup_{s=1}^n U_s$ be a decomposition into disjoint nonempty connected closed  subsets, and  \mbox{$i_s\col U_s\hookrightarrow X$} $(1\le s\le n)$ the canonical inclusions.

The\/ $U_{s}$ are then connected components of\/ $X\<,$ and there is an orthogonal decomposition into indecomposable $A$-rigid complexes$\>:$\vspace{1pt} $F\simeq\bigoplus_{s=1}^n\>(i_s)_{*}(i_s)^*\<\<A$.

If\/ $F\simeq  F_1\oplus \cdots\oplus F_r$ is a direct decomposition with each $F_{t}$ indecomposable, then $r=n$ and $(\<$after renumbering\/$)$ there is for each $s$ an isomorphism $F_{s}\simeq (i_s)_{*}(i_s)^*\<\<A$.
\end{subcorollary}

\begin{proof} 
Since by Theorem~\ref{thm:global rigidity}(iii), $\Supp_{\<X}\<\<F$ is open and closed in $X\<$, therefore each~$U_{s}$ is a connected component of $X\<$. Moreover, if 
$i\colon \Supp_{\<X}\<\<F\hookrightarrow X$ is the inclusion, then $i^*\<\<A$ is semidualizing (Corollary~\ref{lift semid}), and compatibility of $\R\mathcal Hom$ with open immersions
(to see which, use  \cite[2.4.5.2]{Lp2}) implies that $i^*\<\<F$ is $i^*\<\<A$-rigid. It follows then from Theorem~\ref{thm:global rigidity}(iii) that we may assume $F=A$.
\vspace{1pt}

The decomposition $X =\bigsqcup_{s=1}^n U_s$ now yields a decomposition of  $F\in\D(X)$:
\[
F\simeq\bigoplus_{s=1}^n\>(i_s)_{*}(i_s)^{*}\<F=\bigoplus_{s=1}^n\>(i_s)_{*}(i_s)^{*}\<\<A\,.
\]
As before, $(i_s)^{*}\<A$ is a semidualizing complex of ${\mathcal O}_{U_{s}}$-modules,
so  its support is $U_{\<s\>}$, and it is indecomposable; see
Lemma~\ref{semid-prop}. Hence $(i_s)_{*}(i_s)^{*}\<A$ is indecomposable, and has support~$U_{s}$. It then follows from \eqref{Supp Tensor} that the
decomposition above is orthogonal. Moreover,  the complexes
$(i_s)_{*}(i_s)^{*}A$ are $A$-rigid; see Definition \ref{U-canonical}.

Let $F\simeq  F_1\oplus \cdots\oplus F_r$ be a direct decomposition. 
 It results from Lemma~\ref{semid-prop} that this
decomposition is orthogonal. Hence $X=\Supp_{\<X}\<F =\bigsqcup_{t=1}^r V_t$.
Furthermore, $F\in\dcatc X\implies F_t\in\dcatc X$ for all~$t$.
Hence $V_{t}= \Supp_{\<X}\<F_t$ is open and closed;
 and since $F_t$ is indecomposable,  $V_t$ is connected. 
Thus the $V_t$ are the connected components of $X$.  In particular, $r=n$, and, after renumbering, one may assume $V_{t}=U_{t}$ for each $t$. It remains to observe that $F_{s}\simeq (i_s)_{*}(i_{s})^{*}\<F\simeq (i_s)_{*}(i_{s})^{*}\<A$.
\end{proof}

\subsection{Morphisms of rigid complexes}
\label{Rigid global complexes-II}

We present elaborations of Theorem~\ref{thm:global rigidity}, leading
to a simple description of the skeleton of the category of  rigid pairs;
see Theorem~\ref{rigid and clopen} and Remark~\ref{qinverse}.

The result below involves the $\textup{H}^0(X,\OX\<)$ action on $\D(X)$
described in \ref{Gamma acts}.

\begin{subtheorem}
\label{unique iso}
If\/ $(F\<,\rho),$ $(F'\<,\>\rho')$ are $A$-rigid pairs with\/
$\Supp_{\<X}\<\<F=\Supp_{\<X}\<\<F'$ then there exists a unique
isomorphism $(F\<,\rho)\iso(F'\<,\>\rho')$. In particular, any $A$-rigid
pair\/  $(F\<,\rho)$ admits a unique isomorphism into a  $U\<\<$-canonical
one, for some open-and-closed\/ $U\subseteq X,$ necessarily the support
of\/~$F\<$.

Moreover, if $F'=F$ then with\/ $U_{\<\<F}\!:=\Supp_{\<X}\<\<F\<,$  there
is a unique unit\/ $u$ in the ring\/ $\textup{H}^0(U_{\<\<F},\mathcal
O_{U_{\!F}}\<)$ such that $\rho'=\rho\>\bar u$, where\/ $\bar
u\in\textup{H}^0(X,\OX\<)$ is\/ $u$ extended by\/~$0,$ and the unique
isomorphism $(F\<,\rho)\iso(F\<,\>\rho')$ is multiplication in $F$ by\/
$\bar u$.

For any endomorphism\/ $\phi$ of the $A$-rigid pair\/
$(F\<,\rho)$\vspace{.5pt} there is a uniquely determined idempotent
$u\in\textup{H}^0(U_{\<\<F},\mathcal O_{U_{\!F}}\<)$ such that\/ $\phi$~is
multiplication by\/~$\bar u$.
 \end{subtheorem}

\begin{proof}
Modulo Theorem~\ref{thm:global rigidity}, the proof is basically that
of~\cite[7.3]{AIL}. Indeed, Theorem~\ref{thm:global rigidity}(iii)
implies  that we may assume $F=F'$, and that furthermore, we may replace
$X$ by $U\<$, i.e., assume $F=A$ (so that $\bar u  =u$).

Each endomorphism of $F$ is multiplication by a unique element $u$
in~$\textup{H}^0(X,\OX\<)$. {From} Lemma~\ref{mult and RHom} it follows
that multiplication by $u$ induces multiplication by~$u$ on~$F^\dagger$
and multip\-lication by $u^2$ on $\RH_{\<X}(F^\dagger\<,F)$. With
$u^{}_{\<F}$, resp.~$u^{}_{\<H}$, multiplication by $u$
on~$F$, resp.~on $\RH_{\<X}(F^\dagger\<,F)$, we have then that
$u^{}_{\<H}\rho=\rho\> u^{}_{\<F}$, see ~\ref{Gamma acts}, so that
$u_{\<H}^2\rho=u^{}_{\<H}\<\rho\mkern1.5mu u^{}_{\<F}=\rho\>u_{\<F}^2$.

In view of this identity, one gets that $u^{}_{\<F}$ is an isomorphism
from the  rigid pair $(F\<,\rho)$ to the rigid pair~$(F\<,\rho')$ $\iff$
$\rho'u^{}_{\<F}=u_{\<H}^2\rho$ $\iff$ $\rho'u^{}_{\<F}=\rho\> u_{\<F}^2$
$\iff$ $\rho'=\rho\>u^{}_{\<F}$. Thus the sought-after $u$ is the unique
one such that $u^{}_{\<F}$ is the automorphism $\rho^{-1}\rho'$.

In the same vein, when $u^{}_{\<F}$ induces an endomorphism of the rigid
pair $(F\<,\rho)$ one gets a relation $\rho\> u=\rho\> u^2$, whence,
$\rho$ being an isomorphism, $u^2= u$.
 \end{proof}

\begin{subcorollary}
\label{rigid autos}
For any\/ $A$-rigid complex\/ $F\<$, the group of automorphisms
of\/~$F$  acts faithfully and transitively on the set of rigidifying
isomorphisms\/~$\rho$ of\/ $F$.
 \qed
\end{subcorollary}

\begin{subcorollary}
\label{unique map}
If\/ $X$ is connected then every nonzero morphism of\/ $A$-rigid pairs
is an isomorphism.  \qed
  \end{subcorollary}

 \begin{subdefinition}
For any $\D(X)$-map $\phi\col F\to F'$ of $A$-rigid pairs,
$\Supp_{\<X}\<\phi$ is the union of those connected components of $X$
to which the restriction of $\phi$ is nonzero.

By Corollary~\ref{unique map}, if $X$ is connected then  nonzero maps of
$A$-rigid pairs are isomorphisms. So for a composable pair $(\phi,\psi)$
of maps of $A$-rigid pairs,
\begin{equation}
\label{suppcomp}
\Supp_{\<X}(\phi\>\psi)=\Supp_{\<X}\<\phi\cap \Supp_{\<X}\<\psi.
\end{equation}
\end{subdefinition}

\pagebreak[3]

\begin{subcorollary}
 \label{factorization} 
Let\/ $(F\<,\rho)$ and\/ $(F'\<,\rho')$ be  $A$-rigid pairs.
\begin{enumerate}[\quad\rm(1)]
 \item
Suppose that\/ $\Supp_{\<X}\<F\subseteq \Supp_{\<X}\<F'\<$. Then there
is a unique monomorphism\/ $(F\<,\rho)\hra(F'\<,\rho')$ and a unique
epimorphism $(F'\<,\rho')\tra(F\<,\rho)$.
 \item
For any morphism $\phi\col(F\<,\rho)\to(F'\<,\rho'),$
if\/ $(G,\sigma)$ is an\/ $A$-rigid pair with\/
$\Supp_{\<X}\<G=\Supp_{\<X}\<\phi$ then\/ $\phi$ factors uniquely as
$$
\smash{(F\<,\rho)\overset{\lift1.3,\phi'\>,}
\tra(G\<,\sigma)\overset{\lift1.3,{\,\,\phi''},}\hra(F'\<,\rho')}
$$
with\/ $\phi'$ an  epimorphism and\/ $\phi''$ a monomorphism.
 \end{enumerate}
Thus\/ $\phi$ is uniquely determined by its source, target and support.
\end{subcorollary}

\begin{proof} 
Looking at connected components separately, one reduces to where $X$
is connected; the assertions then follow from Corollary~\ref{unique map}
and Theorem~\ref{unique iso}.
 \end{proof}

Here is a \emph{structure theorem} for the category $\rp AX$  of\/
$A$-rigid pairs.

\begin{subtheorem}
\label{rigid and clopen} 
Let\/ $\oc X$ be the category whose objects are the open-and-closed
subsets of\/ $X,$ and whose maps\/ $U\to V$ are the open-and-closed
subsets of\/~\mbox{$U\cap V,$} the composition of\/ $S\subseteq U\cap V$
and\/ $T\subseteq V\cap W$ being\/ $S\cap T\subseteq U\cap W\<$.

Let\/ $\Psi\col\rp AX\to\oc X$ be the functor taking\ $(F\<,\rho)\in\rp
AX$ to\/ $\Supp_{\<X}\<F\<,$ and taking a morphism\/ $\phi\in\rp AX$
to\/ $\Supp_{\<X}\<\phi$ $($see \eqref{suppcomp}$)$.

This\/ $\Psi$ is an equivalence of categories. 
\end{subtheorem}

\begin{proof}
Let\/ $(F\<,\rho)$ and\/ $(F'\<,\rho')$ be $A$-rigid
pairs, $U\!:=\Supp_{\<X}\<F\<$, $V\!:=\Supp_{\<X}\<F'\<$,
and $S$ an open-and-closed subset of $U\cap V$. It follows
from Corollary~\ref{factorization}, with $(G\<,\sigma)$ the
$S$-canonical pair, that there is a unique map of $A$-rigid pairs
$\phi\col(F\<,\rho)\to(F'\<,\rho')$ such that $\Supp_{\<X}\<\phi=S$,
whence the conclusion.
  \end{proof}

\begin{subremark}
\label{qinverse}
A quasi-inverse\/ $\Phi$ of\/ $\Psi$ can be constructed as follows:

$\Phi\col\oc X \to\rp AX$  takes an open-and-closed\/ $U\subseteq X$ to an
arbitrarily chosen rigid pair\/ $(F\<,\rho)$ with $\Supp_{\<X}\<F=U;$ and
then for any\/ $\oc X$-map $S\subseteq U\cap V,$ $\Phi(S)$ is  the unique
epimorphism\/ $\Phi U\tra \Phi S$ followed by the unique monomorphism\/
$\Phi S\hra \Phi V$ (see Corollary~\ref{factorization}).

That this describes a functor is, modulo \eqref{suppcomp}, straightforward to see.

Taking into account that the map $S\subseteq U\cap V$ factors as a split
epimorphism (namely $S\subseteq U\cap S)$ followed by a split monomorphism
(namely $S\subseteq S\cap V)$, and that any functor respects left and
right inverses, one sees that in fact \emph{all} quasi-inverses of~$\Psi$
have the preceding form.

In particular, there is a canonical $\Phi$, associating to each $U$
the $U$-canonical pair. Thus $\oc X$ is \emph{canonically isomorphic}
to the category of canonical $A$-rigid pairs.
 \end{subremark}

The next result is in preparation for establishing a gluing property
for rigid pairs.

\begin{sublemma}
\label{f*rigid}
If $g\col Z\to X$ is a perfect map and $F$ is an $A$-rigid complex in
$\D(X)$, then\/ $\bL g^*\<\<A\in\dcatc Z$ is semi\-dualizing and\/
$\bL g^*\<\<F$ is\/ $\bL g^*\<\<A$-rigid.
 \end{sublemma}

\begin{proof}
That $\bL g^*\<\<A$ is semidualizing is given by Corollary~\ref{lift semid}.

If $\rho$ is an $A$-rigidifying isomorphism for $F\in\D(X)$, then, abusing notation, we let 
$\bL g^*\<\<\rho$ be the composed isomorphism
\begin{align*}
\bL g^*\<\<F&
\iso \bL g^*\RH_{\<X}(\bigcheck F\<,F)\\
&\iso 
\RH_{Z}(\bL g^*\<\<\bigcheck F\<,\>\bL g^*\<\<F)\\
&\iso
\RH_{Z}(\RH_{Z}(\bL g^*\<\<F\<,\>\bL g^*\<\<A),\>\bL g^*\<\<F),
\end{align*}
where the first isomorphism is the result of applying the functor $\bL
g^*\<\<$ to $\rho$, and the other two come from \eqref{first_iso}. Thus
$\bL g^*\<\<\rho$ is $\bL g^*\<\<A$-rigidifying  for $\bL g^*\<\<F\<$.
 \end{proof}

\begin{subtheorem}
\label{gluing} 
Let\/ $g\col Z\to X$ be a faithfully flat scheme-map,
\mbox{$W\!:=Z\times_X Z,$} $\pi_1\col W\to Z$ and\/ $\pi_2\col W\to Z$
the canonical projections.

Let\/ $A\in \D(X)$ be  semidualizing. If\/ $(G,\sigma)$ is
a\/ $g^*\<\<A$-rigid pair such that there exists an isomorphism
$\pi_1^*G\simeq\pi_2^*G,$ then there is, up to unique isomorphism,
a~unique\/ $A$-rigid pair\/~$(F\<,\rho)$ such that\/ $(G,\sigma)\simeq
(g^*\<\<F\<,\>g^*\<\<\rho)$.
 \end{subtheorem}

\begin{proof} (Uniqueness.) If 
$g^*\<\<F\simeq g^*\<\<F'$ then, since
$$
g^{-1}\Supp_{\<\<X}F =\Supp_{\<Z} g^*\<\<F=\Supp_{\<Z} g^*\<\<F'=
g^{-1}\Supp_{\<\<X}F',
$$
and $g$ is surjective, therefore $\Supp_{\<\<X}F=\Supp_{\<\<X}F'\<$;
and so by Theorem~\ref{unique iso}, there is a unique isomorphism
$(F\<,\>\rho)\iso(F'\<,\>\rho')$.

(Existence.) In view of Theorem~\ref{thm:global rigidity}, we may assume that  $G=Jg^*\<\<A$ for some idempotent $\mathcal O_{\<Z}$-ideal $J$. Then, for $i=1,2$, Corollaries~\ref{idem and Hom} and~\ref{lift mult idem} yield 
\begin{align*}
\Supp_W \pi_i^*G&=\Supp_W (\pi_i^*\<\<J\dtensor W\pi_i^*g^*\<\<A)\\
&=\Supp_W\pi_i^*\<\<J\cap\>\Supp_W \pi_i^*g^*\<\<A\\
&=\Supp_W \pi_i^*\<\<J.
\end{align*}
So $\pi_1^*J$ and $\pi_2^*J$, being isomorphic to idempotent ideals with the same support, must be isomorphic.
Hence by Proposition~\ref{idem gluing}, there is a unique idempotent $\OX$-ideal
$I$ such that $J=I\mathcal O_{\<Z}$. If $F=IA$ then $G\simeq g^*\<\<F$. 

Let $\rho$ be a rigidifying isomorphism for $F\<$, so that
$(g^*\<\<F\<,\>g^*\<\<\rho)$ is a $g^*\<\<A$-rigid pair. By Theorem~\ref{unique iso}, there is a unique
isomorphism $(g^*\<\<F\<,\>g^*\<\<\rho)\iso(G\<,\>\sigma)$.
\end{proof}

\subsection{Relative rigidity}
\label{Relatively rigid complexes}

With reference to a G-perfect map $f\col X\to
Y\<$, we  take particular interest in those complexes
that are $f^!\OY$-rigid---complexes we will simply call
$f\<$-\emph{rigid}.\vspace{1pt}

For $g$ any essentially \'etale map (so that, by Proposition~\ref{compn},
$f\<g$ is G-perfect), there is a natural isomorphism of
functors $(f\<g)^!\simeq g^*\<\<f^!$ (see \S\ref{twisted inverse}). By
Lemma~\ref{f*rigid},  if $P$ is $f$-rigid then $g^*\<\<P$ is
$(f\<g)$-rigid.

The following \emph{\'etale gluing} result (where for simplicity we
omit mention of rigidi\-fying isomorphisms) is an immediate consequence
of Theorem~\ref{gluing}.

\begin{subproposition}
\label{etale gluing} 
Let\/ $Z\xra{g}X\xra{f}Y$ be scheme-maps, where \/ $f$ is G-perfect
and $g$ is essentially \'etale and surjective. Let\/
$W\!:=Z\times_X Z,$  with canonical projections\/ $\pi_1\col W\to Z$
and\/ \mbox{$\pi_2\col W\to Z$.} If\/ $P$ is an\/ $(f\<g)$-rigid complex
such that there exists an isomorphism\/  $\pi_1^*P\simeq\pi_2^*P,$ then
there exists, up to  isomorphism, a~unique $f\<$-rigid complex\/~$F$
with\/ $g^*\<\<F\simeq P\<$.\qed \end{subproposition}

Fix a semidualizing complex~$A$ on a scheme $X$. The main result in
this section, Theorem~\ref{extend functors}, is that for any additive
functor from $A$-rigid complexes to the derived category of some
scheme, that takes $A$ to a semidualizing complex $C$---and hence,
by Theorem~\ref{thm:global rigidity}(iv), takes $A$-rigid complexes
to $C$-rigid complexes---there is a unique lifting to the category of
$A$-rigid pairs that takes the canonical pair $(A,\rho^A)$ to $(C,
\rho^{C}\>)$, provided that the functor ``respects intersection of
supports."

{From} Theorem~\ref{extend functors} we will derive the behavior
of relatively rigid complexes with respect to perfect maps
(Corollaries~\ref{rigid functors} and~\ref{g! rigid}). These results
generalize---and were inspired by---results in \cite[Sections 3 and
6]{YZ}.  

Let $\rc {\>\>A}X\subseteq\D(X)$ be the full subcategory of
$A$-rigid complexes, and  let $\rp {\>A}X$ de the category of $A$-rigid
pairs. Let~$\bph_{\<\<X}\col\rp {\>A}X\to \D(X)$ be the functor taking
$(F\<,\rho)$ to $F\in\rc {\>\>A}X$.  The rigid pair $(A,\rho^A)$ is defined 
in Example~\ref{gcanonical}.

\begin{subtheorem}\label{extend functors}
Let\/ $X$ and\/ $Z$ be schemes, let\/ $A\in\D(X)$ be semidualizing,
and let\/ $\CF\col\rc AX\to\D(Z)$  be an additive functor such that\/
$\CF A$ is semidualizing.\vspace{1pt}

There exists at most one functor\/ $\oCF\col\rp AX\to\rp{\CF \<A}Z$,
such that\vspace{1pt} 
  \[
\smash{\bph^{}_{\<\<Z}\oCF=\CF\bph^{}_{\<\<X}}
  \quad\text{and}\quad
\oCF(A,\rho^A)=(\CF A,\rho^{\>\CF \<A})\,.
  \]
For such an\/ $\oCF$  to exist it is necessary that for any 
idempotent\/ $\OX\<$-ideals\/ $I\<,\>J,$\vspace{-1pt}
\begin{equation}
\label{keep cap}
\Supp_Z\CF(IJA)=\Supp_Z\CF(IA)\cap \>\Supp_Z\CF(JA),
\end{equation}
and it is sufficient that\/ \eqref{keep cap} hold whenever\/ $IJ=0$.
\end{subtheorem}

\begin{subremark}\label{fab}
Let $a,b\in\textup{H}^0(X\<,\OX\<)$ be the idempotents such that $I=a\OX$
and $J=b\OX$. Since $IA$ admits a monomorphism into $A$, therefore
$\CF(IA)$ admits a  monomorphism into $\CF A$, and it follows from
Theorem~\ref{thm:global rigidity} that there is a unique idempotent
$f(a)\in\textup{H}^0(Z, \mathcal O_{\<Z})$ with $\CF(IA)\simeq f(a)\CF
A$. By \eqref{Supp Tensor}, Corollary~\ref{idem and Hom}, and the
fact that a semidualizing complex on a scheme is supported at every
point of the underlying space, see Lemma \ref{semid-prop}, 
condition~\eqref{keep cap} amounts then to $f(ab)=f(a)f(b)$.
  \end{subremark}

Before proving Theorem~\ref{extend functors}, we gather together some
examples.  Part (1) of the next corollary elaborates Lemma~\ref{f*rigid}.

Recall  that if $g\col Z\to X$ is perfect then both $\bL g^*\<\<
B$ and $g^!A$ are semidualizing; see Corollary~\ref{lift semid}. If
$L\in\D(X)$ is invertible then $L\dtensor{\<\<X}A$ is semidualizing, by
Corollary~\ref{refl and inv}(3); and if $F\in\Dqcpl(X)$, then there is
as in~\eqref{!tensoriso} a natural isomorphism $g^!\<L\dtensor{\<Z}\bL
g^*\<\<F\iso g^!(L\dtensor{\<\<X}F)$.

\begin{subcorollary}\label{rigid functors} Let\/ $g\col Z\to X$ be a perfect map, 
\vspace{1pt} and\/ $A\in\dcatc X$ semidualizing.
   \begin{enumerate}[\quad\rm(1)]
     \item
There is a unique functor\/ $g^{**}\col\rp AX\to\rp{\>\> \bL g^*\<\< A}Z$ such that
$$
\smash{\bph^{}_{\<\<Z} \>\>g^{**}=\bL g^* \quad\textup{and}\quad
g^{**}(A,\rho^A)=(\bL g^*\<\<A,\>\rho^{\>\bL g^*\!A}\>).}
$$
     \item
There is a unique functor\/ $g^{!!}\col\rp AX\to\rp{g^!\<\< A}Z$ such that
$$
\smash{\bph^{}_{\<\<Z} \>\>g^{!!}=g^! \quad\textup{and}\quad
g^{!!}(A,\rho^A)=(g^!\<\<A,\>\rho^{\>g^!\!A}\>).}
$$
     \item
For each invertible\/ $L\in\D(X)$ there is a unique bifunctor 
  \begin{align*}
g^\otimes\col \rp{g^!\<\<L\<\<}Z\times \rp {A}X&\to \rp{\>g^! (L\dtensor{\<\<X}A)}Z
  \\
\intertext{such that} 
\bph^{}_{\<\<Z}\>\>g^\otimes(P\<,F)&=P\dtensor{\<Z}\bL g^*\<\<F
  \\
\intertext{and} 
g^\otimes\big((g^!L,\rho^{\>g^!L}),(A,\rho^A)\big)
&=(g^!\<(L\dtensor{\<\<X}A),\>\rho^{\>g^!\<(L\dtensor{\<\<X}A)}\>)\,.
  \end{align*}
     \end{enumerate}
\end{subcorollary}

\begin{proof}
Corollary~\ref{lift mult idem} implies that for either functor, one has
in Remark~\ref{fab} that $f(a)$ is the image of $a$ under the natural
map $\textup{H}^0(X\<,\OX\<)\to \textup{H}^0(Z, \mathcal O_{\<Z})$.
Thus \mbox{$f(ab)=f(a)f(b)$} holds, and so (1) and (2) result from
Theorem~\ref{extend functors}.

For (3) replace $X$ in Theorem~\ref{extend functors} by the disjoint
union $Z\sqcup X$. For \mbox{$P\in\D(Z)$} and $F\in\D(X),$ let
$(P\<,F)\in\D(Z\sqcup X)$ be the complex whose restriction to $Z$ is
$P$ and to $X$ is $F$. There is an obvious functor $\CF\col\D(Z\sqcup
X)\to\D(Z)$ taking $(P\<,F)$ to $P\dtensor{\<Z}\bL g^*\<\<F$.   This
functor takes the semidualizing complex $(g^!L,A)$ to the semidualizing
complex $g^!L\dtensor{\<Z}\bL g^*\<\<A\simeq g^!(L\dtensor{\<\<X}\<A)$.
Using \eqref{Supp Tensor} and~Remark~\ref{fab}, one verifies that
\eqref{keep cap} holds; and so (3) results.
  \end{proof}

Recall that if $Z\xra{g}X\xra{f}Y$ are maps such that\/ $g$ is
perfect and $f$ is G-perfect then $f\<g$ is G-perfect
(Proposition~\ref{compn}).  Taking $A=f^!\OY$ and $L=\OX$ in (2) and~(3)
of Corollary~\ref{rigid functors} one gets:

\begin{subcorollary}\label{g! rigid}
Let\/ $g\col Z\to X$ be perfect, and\/ $f\col X\to Y$  G-perfect.
   \begin{enumerate}[\quad\rm(1)]
     \item
If\/ $F$ is\/ $f\<$-rigid then\/ $g^!\<\<F$ is $f\<g$-rigid.
     \item
If\/ $P$ is\/ $g$-rigid and\/ $F$ is\/ $f\<$-rigid then\/
$P\dtensor{\<Z} \bL g^*\<\<F$ is\/ $f\<g$-rigid.
\qed
   \end{enumerate}
\end{subcorollary}

\begin{subcorollary}\label{g_*rigid} 
Let\/ $g\col Z\to X$ be a proper map such that the natural map is an
isomorphism\/ $\OX\iso\R g_*\mathcal O_{\<Z}$.  Let $A\in\Dqcpl(X)$
be such that \smash{$g^!A$} is semidualizing.

Then\/ $A$ is semidualizing, the canonical  map is an
isomorphism $\R g_*g^!\<\<A\iso A,$ and there is a unique functor
$g_{**}\col\rp{g^!\<\<A}Z\to \rp AX$ such that
$$
\smash{\bph^{}_{\<\<X} \>\>g_{**}=\R g_*\bph^{}_{\<\<Z} \quad\textup{and}\quad
g_{**}(g^!\<\<A,\>\rho^{g^!\!A})=(\R g_*g^!\<\<A,\>\rho^{\R g_*g^!\!A}).}
$$
Hence, if\/ $f\col X\to Y$ is such that\/ $f\<g$ is G-perfect
then $f$ is G-perfect, and if\/ $P$ is $f\<g$-rigid then $\R
g_*P$ is $f$-rigid.  \end{subcorollary}

\begin{proof}  
That $A$ is semidualizing is given by Proposition~\ref{dir image and
refl}.

There are, for $E\in\Dqc(X)$, natural isomorphisms, the second from
~\ref{twisted inverse}(ii), and the third from \eqref{projection},
\begin{align*}
\textup{Hom}_{\>\mathsf D(X)}(E,\R g_*g^!\<\<A)&\cong \textup{Hom}_{\>\mathsf D(Z)}(\bL g^*\<\<E\<,\> g^!\<\<A)\\
&\cong
\textup{Hom}_{\>\mathsf D(X)}(\R g_*(\mathcal O_{\<Z} \dtensor Z{\bL g^*\<\<E)}, A) \\
&\cong 
\textup{Hom}_{\>\mathsf D(X)}(\R g_*\mathcal O_{\<Z} \dtensor{\<\<X} E\<, A)
\cong  \textup{Hom}_{\>\mathsf D(X)}(E\<, A).
\end{align*}
It follows, via \cite[3.4.7(ii)]{Lp2}, that the canonical map is an isomorphism 
\[
\R g_*g^!\<\<A\iso A.
\]
By assumption, one has the natural isomorphism
$\textup{H}^0(X\<,\OX\<)\iso\textup{H}^0(Z,\mathcal O_{\<Z})$. So there
is a bijection between the idempotents in these two rings; and also, $g$
is surjective. Hence $g^{-1}$ gives a bijection from the open-and-closed
subsets of~$X$ to the open-and-closed subsets of $Z$. Furthermore, for
any $P\in\dcatc Z$, $\Supp_{\<Z}P$ is closed, whence, $g$ being proper,
$U\!:=X\setminus g(\Supp_{\<Z}P)$ is open; and  the restriction of $P$ to
$g^{-1}U$ is acyclic. Thus $\Supp_{\<\<X}\< \R g_*P\subseteq g(\Supp_{\<Z}
P).$ The verification of~\eqref{keep cap}, with $\CF=\R g_*$ and $A$
replaced by $g^!\<\<A$, when $IJ=0$---so that $\Supp_{\<Z}(Ig^!\<\<A)$
and $\Supp_{\<Z}(Jg^!\<\<A)$ are disjoint open-and-closed subsets
of~$Z$---is now immediate. The existence and uniqueness of $g_{**}$
follows then from Theorem~\ref{extend functors}.

For the last assertion, take $A=f^!\OY$.
\end{proof}

\begin{subcorollary}
\label{rigidity and base change}
Let there be given  a tor-independent fiber square 
$(\<$see~\textup{\ref{fiber square})}
\[
\xymatrixrowsep{1pc}
\xymatrixcolsep{1pc}
\xymatrix{
X'\:\ar@{->}[rr]^-{v}
\ar@{->}[dd]_-{h}
&&
\:X
\ar@{->}[dd]^-{f}
\\
&{}
\\
Y'\:
\ar@{->}[rr]_-{u}
&&
\:Y
}
\]
in which\/ $f$ is G-perfect. 

If\/ $u$ is flat, or if\/ $u$ is perfect and\/ $f$ is proper, then\/ $h$
is G-perfect and for any\/ $f\<$-rigid\/ $\OX\<$-complex\/~$F\<,$
$\bL v^*\<\<F$ is\/ $h$-rigid.  \end{subcorollary}

\begin{proof} 
Proposition~\ref{G-dim and base change} and \cite[p.\,245,
3.5.2]{Il} imply $h$ is G-perfect and $v$ is
perfect. By Corollary~\ref{rigid functors}(i), $\bL v^*\<\<F$ is $\bL
v^*\<\<f^!\OY$-rigid, i.e., $h^!\mathcal O_{Y'}$-rigid; see \eqref{torindt
basechange}.  \end{proof}

\begin{proof}[Proof of Theorem \emph{\ref{extend functors}}](Uniqueness.)  Let $(G\<,\sigma)$ be an $A$-rigid pair. 

Set $(\CF
G\<,\tau)\!:=\oCF(G\<,\sigma)$. Let $\phi^{}_{\<G}$~be the  unique (split)
monomorphism from $(G\<,\sigma)$ to the canonical pair $(A,\rho^A)$, so that
$\oCF(\phi^{}_{\<G})$ is a (split) monomorphism,\vspace{1pt} necessarily
the unique one from $(\CF G\<,\tau)$ to $(\CF A,\sigma^{\CF A})$, see
Corollary~\ref{factorization}. It follows then from Lemma~\ref{unique rho}
that $\tau$ depends only on $\CF$ and $(G\<,\sigma)$.

Also, for any morphism $\phi$ of $A$-rigid pairs,
$\smash{\bph^{}_{\<\<Z}\oCF=\CF}$ implies
 $\oCF\phi=\CF\phi$.
\vspace{1pt}

(Necessity of~\eqref{keep cap}). Let $\Psi_{\<\<Z}\col\rp {\CF
A}Z\to\oc Z$ be as in Theorem~\ref{rigid and clopen}. Let $\Phi:\oc
X\to\rp AX$ be as in Remark~\ref{qinverse}, sending an open-and-closed
$U\subseteq X$ to~$I_UA$, where $I_U$ is the idempotent $\OX$-ideal
that is $\mathcal O_U$ over $U$ and (0)~elsewhere.  Then
$\Psi_{\<\<Z}\oCF\>\Phi\col\oc X \to \oc Z$  respects composition of maps,
i.e., \eqref{keep cap} holds.\looseness=-1 \vspace{1pt}

(Existence.) Since any functor preserves a map's property of being
split---mono or~ epi---Theorem~\ref{thm:global rigidity}(iv) shows
that $\CF$ takes $A$-rigid complexes to $\CF A$-rigid complexes; and the
preceding uniqueness argument shows how $\oCF(G,\sigma)$ must be defined. It
remains to prove that for any morphism $\phi\col(G\<,\sigma)\to(G'\<,\sigma')$
of $A$-rigid pairs, $\CF\phi$~is a morphism of $\CF A$-rigid pairs.

Let $U_1,\dots,U_n$ be the connected components of $X$. For each $j$,
let $V_j$ be the support of the $\CF A$-rigid complex $\CF(I_{U_j}A)$
(see above).  The condition~\eqref{keep cap}, for~$IJ=0$, guarantees
that if $j\ne k$ then the open-and-closed subsets $V_j$ and $V_k$ are
disjoint. So we need only show that

\noindent$(*)$ \emph{the
restriction of\/ $\CF\phi$ over each\/ $V_j$ is a morphism of\/ $\CF
A|_{V_j}$-rigid pairs.}

Corollary~\ref{rigid decomp} shows that $\phi=\sum_{j=1}^n\phi_j$
where for each $j$, the source and target\- of $\phi_j$ each have
support that, if not empty, is~ $U_j$.  Then, since $\CF$ is additive,
$\CF\phi=\sum_{j=1}^n\CF\phi_j$; and the source and target of $\CF\phi_j$
each have support contained in~$V_j$ (see the first assertion in
Theorem~\ref{unique iso}). Hence the restriction of $\CF\phi$ over~$V_j$
is~$\CF\phi_j$.  Proving $(*)$ is thus reduced to the case where $X$
is connected, so that by~Corollary~\ref{unique map}, $\phi$ is either
0 or an isomorphism.

If $\phi=0$, $(*)$ is obvious.  If $\phi$ (hence
$\CF\phi$) is an isomorphism consider the diagram, where $(\CF
G\<,\tau)\!:=\oCF(G\<,\sigma)$, $(\CF G'\<,\>\tau')\!:=\oCF(G'\<,\>\sigma')$,
where $\phi^{}_{\<G'}$ is as above, and where the maps on the right are
induced by those on the left:
$$
\CD
\CF G @>\tau >> \RH_{\<Z}(\RH_{\<Z}(\CF G\<,\CF A),\>\CF G) \\
@V\CF \phi VV  @VV\xi V\\
\CF G' @>\tau' >> \RH_{\<Z}(\RH_{\<Z}(\CF G'\<,\CF A),\>\CF G') \\
@V\CF \phi^{}_{\<\<G'}VV  @VV\xi' V\\
\CF A @>>\sigma^{\CF\<\< A}> \RH_{\<Z}(\RH_{\<Z}(\CF A,\CF A),\>\CF A)
\endCD
$$
By the above-indicated definition of $\tau$ and $\tau'\<,$ the bottom
square commutes, as does the square obtained by erasing $\tau'$.
Since $\xi'$ is a monomorphism, therefore the top square commutes
too. Thus $\CF \phi$ is a map of $\CF A$-rigid pairs.  
\end{proof}

\begin{subremark}
One would naturally like more concrete definitions of the functors in
Corollary~\ref{rigid functors}.  

One does find in \cite[\S3]{YZ} some explicitly formulated---in DGA
terms---versions of special cases of these functors. (Indeed, that's
what suggested Corollary~\ref{rigid functors}.) But getting from here
to there does not appear to be a simple matter. One might well have to
go via the Reduction Theorem \cite[4.1]{AILN}, the main result of that
paper, cf.\ \cite[8.5.5]{AIL}); and, say for smooth maps, make use of
nontrivial formal properties of Verdier's isomorphism (\S\ref{smooth ^!}).

In Duality Land the well-cultivated concrete and abstract plains are 
not presently known to be connected other than by forbidding 
mountain passes, that can only be traversed  by hard slogging.

 \end{subremark}

\section*{Background}
\numberwithin{equation}{theorem}

We review background concepts and basic facts having to do with
\mbox{scheme\kern.5pt-maps,} insofar as needed in the main text.
Of special import is the \emph{twisted inverse-image pseudofunctor},
a~fundamental object in Grothendieck duality theory.

\emph{Rings and schemes are assumed throughout to be noetherian.}

\appendix

\section{Essentially finite-type maps}
\label{eft etc}
 
  \begin{chunk}
  \label{reduce1}
A homomorphism $\sigma\col K\to S$ of commutative rings
is \emph{essentially of finite type} if $\sigma$
can be factored as a composition of ring-homomorphisms
  \[
K\hra K[x_1,\dots,x_d]\to V^{-1}K[x_1,\dots,x_d]\tra S\,,
  \]
where $x_1,\dots,x_d$ are indeterminates, $V\subseteq K[x_1,\dots,x_d]$
is a multiplicatively closed set, the first two maps are canonical  and
the third is surjective.  The map $\sigma$ is \emph{of finite type}
if one can choose $V=\{1\}$; the map $\sigma$ is \emph{finite} if it
turns $S$ into a finite (that is, finitely generated) $R$-module.

A homomorphism $\dot\sigma\col K\to P$ is (\emph{essentially})
\emph{smooth} if it is flat and (essentially) of finite type,  and if for
each homomorphism of rings $K\to k$, where $k$ is a field, the ring $k
\otimes_KP$ is regular. By \cite[17.5.1]{Gr4}, this notion of smoothness
is equivalent to the one defined in terms of lifting of homomorphisms.

When $\dot\sigma$ is essentially smooth the $P$-module
$\Omega_{\dot\sigma}$ of relative K\"ahler differentials is finite
projective; we say $\dot\sigma$ has \emph{relative dimension} $d$ if for
every $p\in\Spec S$, the free $S_p$-module $(\Omega_{\dot\sigma}\<)_p$
has rank~$d$.
  \end{chunk}

  \begin{chunk}
  \label{reduce2}
A scheme\kern.5pt-map  $f \col X \to Y$ is \emph{essentially of
finite type\/} if every~$y \in Y$ has an affine open neighborhood $V =
\Spec(A)$ such that $f^{-1}V$ can be covered by finitely many affine
open sets $U_i = \Spec(C_i)$ so that the corresponding ring homomorphisms
$A \to C_i$ are essentially of finite type.  

If, moreover, there exists
for each $i$ a  multiplicatively closed subset $V_i\subseteq A$ such
that  $A\to C_i$ factors as $A\to V_i^{-1}A\iso C_i$ where the first
map is canonical and the second is an isomorphism (in other words,
$A\to C_i$ is a \emph{localization} of $A$), then we say that $f$
is  \emph{localizing}.  If the scheme-map $f$ is localizing and also 
set-theoretically injective, then we say that $f$ is a \emph{localizing 
immersion}.

The map $f$ is \emph{essentially smooth} (of relative dimension $d\>$)
if it~is essentially of finite type and the above data $A\to C_i$ can be
chosen to be essentially smooth ring homomorphisms (of relative dimension
$d\>$).  The map $f$  is \emph{essentially \'etale} if it is essentially
smooth  of relative dimension 0.  Equivalently, $f$ is essentially smooth
(resp.~\'etale) if it is essentially of finite type and formally smooth
(resp.~\'etale); see~\cite[\S17.1]{Gr4}.  For example, any localizing
map is essentially \'etale.\vspace{1pt}
  \end{chunk}

\begin{remark}\label{diag}
We will refer a few times to proofs in \cite{Il} that make use of the fact that 
the diagonal  of a smooth map is a quasi-regular immersion. To ensure that 
those proofs apply here, we note that  the same property for 
\emph{essentially smooth} maps is given by \cite[16.10.2 and 16.9.4]{Gr4}.
\end{remark}

In \cite[4.1]{Nk2}, extending  a compactification theorem of
Nagata,  Nayak shows that \emph{every essentially-finite-type separated
map $f$ of noetherian schemes factors\- as \mbox{$f=\bar{f\,}\!u$}
with $\bar{f\,}\!$ proper and $u$~a localizing immersion.}\vspace{1.5pt}

\begin{example}
\label{example}(Local compactification.) 
A map $f\col X=\Spec S\to\Spec K=Y\<$ coming from an essentially
finite\kern.5pt-type homomorphism of rings $ K\to S$ factors
as\looseness=-1
 $$
X\xra{j} W \overset i\hookrightarrow  \bar W\xra{\pi}Y,
 $$
where $W$ is the Spec of a finitely-generated $K$-algebra $T$ of which
$S$ is a localization, $j$ being the corresponding map, where $i$ is
an open immersion, and where $\pi$ is a projective map, so that $\pi$
is proper and $ij\>$ is a localizing immersion.
 \end{example}

\section{Review of global duality theory}

 \emph{All scheme-maps are assumed to be essentially of finite type and separated.}

\vspace{2pt}

We recall some global duality theory, referring to \cite{Lp2} and \cite{Nk2} for details.

\begin{chunk}
To any scheme\kern.5pt-map $f\col X\to Y$ one associates the right-derived
direct-image functor~$\Rf\col\Dqc(X)\to\Dqc(Y)$ and its left adjoint, the 
left-derived inverse-image functor $\bL f^*\col\Dqc(Y)\to\Dqc(X)$
\cite[3.2.2, 3.9.1, 3.9.2]{Lp2}. These functors interact with
the left-derived tensor product $\Otimes{}$ via a natural isomorphism
  \begin{equation}
    \label{^* and tensor}
\bL f^*(M\Otimes{Y}N)\iso
\bL f^*\<\<M\Otimes{\<\<X} \bL f^*\<\<N
\qquad\big(M,N\in\D(Y)\big),
   \end{equation}
see \cite[3.2.4]{Lp2};  via the functorial map 
\begin{equation}
\label{_* and Tensor}
\Rf F\Otimes{Y}\Rf G \to \Rf (F\Otimes{\<\<X} G)
\qquad\big(F,G\in\D(X)\big)
\end{equation}
adjoint to the natural composite map
$$
\bL f^*(\Rf F\Otimes{Y}\R f_*G)\iso
\bL f^*\Rf F\Otimes{\<\<X} \bL f^*\Rf G\longrightarrow
F\Otimes{\<\<X} G;
$$
and via the \emph{projection isomorphism}
\begin{equation}
\label{projection}
\Rf F\Otimes{Y} M \iso \Rf(F\Otimes{\<\<X} \bL f^*M)
\qquad\big(F\in\Dqc(X),\;M\in\Dqc(Y)\big),
\end{equation}
defined qua map to be  the natural composition
$$
\Rf F\Otimes{Y}M\to \Rf F\Otimes{Y}\Rf\bL f^*M\to
\Rf(F\Otimes{\<\<X} \bL f^*M).
$$
see \cite[3.9.4)]{Lp2}. The projection isomorphism yields a natural isomorphism
\begin{equation}
\label{projection2}
\R f_*\bL f^*\<\<M\simeq \R f_*(\OX\dtensor {\<X} \bL f^*\<\<M)
\simeq  \R f_*\OX \dtensor{\<\<Y}M.
\end{equation}

Interactions with the derived (sheaf-)homomorphism functor $\RH{}$
occur via natural bifunctorial maps:
\begin{equation}
  \label{^* and Hom}
\bL f^*\RH_Y(M,N)\to \RH_{\<X}(\bL f^*\<\<M\<,\> \bL f^*\<\<N)
\qquad\big(M,N\in\D(Y)\big)\,,
\end{equation}
(see \cite[3.5.6(a)]{Lp2}) which is an \emph{isomorphism} if $f$ is an open immersion \cite[p.\,190, end of \S4.6]{Lp2}; and 
\begin{equation}
\label{_* and Hom}
\Rf\RH_{\<X}(F,G)\to\RH_Y(\Rf F\<,\>\Rf G)\qquad\big(F,G\in\D(X)\big),
\end{equation}
the latter corresponding via \eqref{adjunction2} to the natural composition
$$
\Rf\RH_{\<X}(F,G)\Otimes{Y}\Rf F\to \Rf\big(\RH_{\<X}(F,G)\Otimes{\<\<X}F\big)
\xra{\Rf\varepsilon} \Rf G,
$$
where the first map comes from~\eqref{_* and Tensor}, and
$\varepsilon$ is the evaluation map  \eqref{evaluation}.
\end{chunk}

\begin{chunk}
\label{fiber square}
For any commutative square of scheme\kern.5pt-maps
  \begin{equation}
    \label{square}
\xymatrixrowsep{1pc}
\xymatrixcolsep{1pc}
\begin{gathered}
\xymatrix{
X'\ar@{->}[rr]^-{v}
\ar@{->}[dd]_-{h}
&&
X
\ar@{->}[dd]^-{f}
\\
&\Xi
\\
Y'
\ar@{->}[rr]_-{u}
&&
Y
}
\end{gathered}
  \end{equation}
one has the  map
$\theta_{\>\Xi}\col \bL u^* \Rf\to \R\>h_*\bL v^*$
adjoint to the
natural composite map
$$
\Rf\longrightarrow\Rf\R v_*\bL v^*\iso
\R u_*\R\>h_*\bL v^*.
$$
When $\Xi$ is a \emph{fiber square} (which means that the map 
associated to $\Xi$
is an isomorphism $X'\iso X\times_Y Y'$), and $u$ is \emph{flat}, then
$\theta_{\>\Xi}$ is an \emph{isomorphism}. In fact, for any fiber square
$\Xi$, $\theta_{\>\Xi}$ \emph{is an isomorphism\/ $\iff \Xi$ is
tor-independent} \cite[3.10.3]{Lp2}.
  \end{chunk}

\begin{chunk}
\label{twisted inverse}
Duality theory focuses on the \emph {twisted inverse-image pseudo\-functor}
$$
f^!\col\Dqcpl(Y)\to\Dqcpl(X),
$$
where  ``pseudofunctoriality" (also known as \mbox{ ``2-functoriality"})\vspace{.5pt}
entails, in addition to functoriality,  a family of functorial\vspace{1.3pt}
isomorphisms $c^{}_{\< g,\mkern1.5mu f}\col (fg)^!\iso g^!f^!$, one for 
each composable pair $Z\xra{\lift.7,g,} X\xra{\lift1.05,f,} Y$, 
satisfying\vspace{.7pt} a natural  ``associativity" property
vis-\`a-vis any composable triple, see, e.g., \cite[3.6.5]{Lp2}.\vspace{1pt}

This pseudofunctor is uniquely determined up to isomorphism by the following
three properties:\vspace{1pt}

(i) If $f$ is essentially \'etale then $f^!$ is the usual restriction functor
$f^*$.

(ii) If $f$ is proper then $f^!$ is right-adjoint to $\Rf\mkern1mu$.

(iii) If in a fiber square $\Xi$ as in \eqref{square} the map $f$ (and hence
$h$) is proper and $u$~is essentially \'etale, then  the functorial 
\emph{base-change map}
\begin{equation}\label{beta}
\beta_{\>\Xi}(M)\col v^*\<\<f^!\<M\to h^!\<u^*M\qquad\big(M\in\Dqcpl(Y)\big),
\end{equation}
defined to be  adjoint to the natural composition
$$
\R\>h_* v^*\mkern-2mu f^!\<M\underset{\lift1.2,\theta_{\>\Xi}^{-\<1},}{\iso}
u^*\Rf f^!\<M
\longrightarrow u^*\<\<M,\\[1.5pt]
$$
is identical with the natural composite \emph{isomorphism}
$$
v^*\<\<f^!\<M= v^!\<f^!\<M
\iso (fv)^!M=(uh)^!M\iso
h^!\<u^!M  =h^!\<u^*M.
$$
For the existence of such a pseudofunctor, see \cite[section 5.2 ]{Nk2}.
  \end{chunk}

  \begin{chunk}\label{any base change}
Nayak's theorem \cite[5.3]{Nk2} (as elaborated in \cite[7.1.6]{Nk}) 
shows that one can associate, in a unique way, to \emph{every} fiber 
square~$\Xi$ as in \eqref{square} with~$u$ (and hence~$v$) flat, a functorial isomorphism 
  \[
\beta_{\>\Xi}(M)\col v^*\<\<f^!\<M\iso h^!\<u^*M \qquad (M\in\Dqcpl(Y))\,,
  \]
equal to~\eqref{beta} when $f$ is proper, and to the natural 
isomorphism \mbox{$v^*\<\<f^*\iso h^*u^*$} when $f$ is essentially \'etale.\looseness=-1
  \end{chunk}

\begin{chunk}\label{smooth ^!}
Generalizing (i)  in \S\ref{twisted inverse}, let $f\col X\to Y$ be essentially smooth, 
so that by \cite[16.10.2]{Gr4} the relative
differential sheaf $\>\Omega_f$  is locally free over $\OX$. 
On any connected component $W$ of~$X\<$, the 
rank of~$\>\Omega_f$ is a constant, denoted $d(W)$.

There is a \emph{functorial isomorphism}\vspace{1pt}
$$
f^!\<M\iso
\Shift^d\>\Omega^d_f\otimes_{\OX} f^* \<\<M\qquad \big(M\in\Dqc(Y)\big),\\[1pt]
$$
with 
$\Shift^d\>\Omega^d_f$ the complex whose restriction to any  $W$
is $\Shift^{d(W)}\lift2,\bwedge{}^{\lift1.7,\!d(W),}_{\lift.4,\>\mathcal O_W,}, \<\<\big(\Omega_f\big|_W\big)$.\vspace{3.5pt}

\noindent($\Shift$ is the usual translation automorphism of $\D(X)$; and $\lift2,\bwedge,$ denotes\vspace{1.5pt} ``exterior power.")\looseness=-1

To prove this, one may assume that $X$ itself  is connected, and set 
$d\!:= d(X)$.  Noting that the diagonal $\Delta\col X\to X\times_YX$
is defined locally by a regular seq\-uence of length $d$ 
(see Remark~\ref{diag}), so that
\mbox{$\Delta^!\mathcal O_{X\times_Y X}\Otimes{}\bL \Delta^* G\cong\Delta^!G$}
for all $G\in\Dqc(X\times_Y X)$ \cite[p.\,180, 7.3]{H}, one can imitate
the proof of \cite[p.\,397, Thm.\,3]{Ve}, where, in view of~(a) above,
one can drop the properness condition and take $U=X\<$, and where 
finiteness of Krull dimension is superfluous.
\end{chunk}

\begin{chunk}\label{sheaf duality}
The fact that $\beta_{\>\Xi}(M)$ in \eqref{beta}  is an isomorphism for
all $M$ whenever
$u$ is an open immersion and
$f$ is proper, is shown in \cite[\S4.6, part
V\kern.5pt]{Lp2} to be equivalent to \emph{sheafified duality,} which is that
\emph{for any proper $f\col X\to Y,$ and any
\mbox{$F\in\Dqc(X),$} $M\in\Dqcpl(Y),$  the natural composition,
in which the first map comes from \textup{\ref{_* and Hom},}}
\begin{equation}\label{duality iso}
\Rf\mathcal Hom_X(F,\,f^!\<M)\to
\R\mathcal Hom_Y(\Rf F,\,\Rf f^!\<M)\to
\R\mathcal Hom_Y(\R f_*F,M),
\end{equation}
\emph{is an isomorphism.}

Moreover, if the proper map $f$ has finite flat dimension, then sheafified duality holds for \emph{all} $M\in\Dqc(Y)$, see \cite[4.7.4]{Lp2}.\vspace{1pt}

If $f$ is a \emph{finite} map, then \eqref{duality iso} with $F=\OX$ determines the functor $f^!$. (See also \cite[\S2.2]{Co2}.) In particular, if $f\col\Spec B\to \Spec A$ corresponds to a finite  ring homomorphism $A\to B$, and $^\sim$ is the standard sheafification functor, then for an $A$-complex N, $f^!(N^\sim)$ is the $B$-complex
\begin{equation}
\label{finite}
f^!(N^\sim) = \R\textup{Hom}_A(B,N)^\sim,
\end{equation}
where $\R\textup{Hom}_A(B,-)$ denotes the right-derived functor of the functor~$\textup{Hom}_A(B,-)$ from $A$-modules to $B$-modules.
\end{chunk}

\section{Idempotent ideal sheaves}
\label{idempotent ideals}

\begin{definition}
\label{defidem}
 Let\/ $(X,\OX\<)$ be a local-ringed space, that is, $X$ is a topological space and\/  $\OX\<$ is a sheaf of commutative rings whose stalk at each point of $X$ is a local ring $($not necessarily noetherian$).$ An $\OX$-ideal is \emph{idempotent} if it is of finite type $($i.e., locally finitely generated$\>\>)$ 
and  satisfies the equivalent conditions in the next proposition.
\end{definition}

\begin{proposition}
\label{char idem} 
Let\/ $(X,\OX\<)$ be a local-ringed space. Consider the following conditions on an\/ $\OX$-ideal\/ $I$.
\begin{enumerate}[\quad\rm(i)]
\item
There is an\/ $a\in H^0(X,\OX\<)$ such that\/ $a^2=a$ and\/ $I=a\OX$.
\item[\quad{\rm(i$'$)}]  
The identity map of $I$ extends to an $\OX$-homomorphism
$\pi\col\OX\to I$.
\item
There is an open and closed $U\subseteq X$, with inclusion, say,
$i\col U\hookrightarrow X,$ and an $\OX$-isomorphism 
$i_*\mathcal O_U\simeq I$.
\item
The\/ $\OX$-module\/ $\OX/I$ is flat.
\item
For all\/ $\OX$-modules\/ $F\<,$ the natural map is an isomorphism $I\otimes_X F\iso IF.$
\item
For all\/ $\OX$-ideals\/ $J,$ $IJ=I\cap J$.
\item
$I^2=I$.
\end{enumerate}
One has the implications
$$
{\rm(i)}\iff(\textup{i}')\iff{\rm(ii)}\implies{\rm(iii)}\iff{\rm(iv)}\iff{\rm(v)}\implies{\rm(vi);}
$$
and if\/ $I$ is of finite type then\/ ${\rm(vi)}\implies{\rm(i).}$
\end{proposition}

\begin{proof} 
(i) ${}\Leftrightarrow{}$(i$'$). If  (i) holds, let $\pi$ be the map taking $1\in \textup{H}^0(X\<,\OX\<)$ to $a$. Conversely, given (i$'$), let $a=\pi(1)$.

(ii)${}\Rightarrow{}$(i). Let $a$ be the global section 
that is 1 over $U$ and 0 over $X\setminus U$.

(i)${}\Rightarrow{}$(vi). Trivial.

(vi)${}\Rightarrow{}$(ii) when $I$ is of finite type (whence 
(i)${}\Rightarrow{}$(ii) always). The support of $I$,
$U\!:=\{\,x\in X\mid I_x\ne 0\,\}$, is closed when $I$ is of finite type. For any 
$x\in U\<$, since $I_x$ is a finitely generated $\mathcal O_{\<\<X\<,\>x}$-ideal such that $I_x=I_x^2$, therefore Nakayama's lemma shows that  $I_x=\mathcal O_{\<\<X\<,\>x}$. So 
$X\setminus U=\{\,x\in X\mid \mathcal O_{\<\<X\<,\>x}/I_x\ne 0\,\}$ is closed, and thus
$U$~is open as well as closed.
Clearly, $I|_U=\mathcal O_U$ and $I|_{\<X\setminus U}=0$, whence
$I\simeq i_*\mathcal O_U$.

(i)${}\Rightarrow{}$(iii). If (i) holds then  the germ of $a$ at 
any $x\in X$ is 1 or 0, so 
$(\mathcal O/I)_x$ is either (0) or $\mathcal O_{\<\<X\<,\>x}$, both of which are flat
over $\mathcal O_{\<\<X\<,\>x}$.\vspace{1pt}

The remaining implications can be tested stalkwise, and so reduce to
the corresponding well-known implications for ideals $I$, $J\>$ in a
local ring $R$, and $R$-modules~$F\>$:\looseness=-1

(iii)${}\Rightarrow{}$(iv). The surjection
$
I\otimes_R F\tra IF\subseteq R\otimes_R F
$ 
has kernel $\textup{Tor}_1^R(R/I,F)=0$.

(iv)${}\Rightarrow{}$(v). $(I\cap J)/IJ$ is the kernel of the natural
injective (by (iv)) map
$$
R/IJ\cong I\otimes_R R/J\to R\otimes_R (R/J)=R/J.
$$

(v)${}\Rightarrow{}$(iii). Flatness of $R/I$ is implied by
injectivity, for all $R$-ideals $J$, of the natural map 
$J/IJ\cong J\otimes_R(R/I)\to R\otimes_R(R/I)=R/I$, with kernel $(I\cap J)/IJ$.

(v)${}\Rightarrow{}$(vi). Take $J=I$.
\end{proof}

\begin{corollary}
\label{idem and clopen}
{\rm(1)} Taking $a$ to $a\OX$ gives a bijection from the set of idempotent elements of $H^0(X,\OX\<)$ to the set of  idempotent\/ $\OX$-ideals.

{\rm(2)} There is a bijection that  associates to each idempotent\/ $\OX$-ideal
its support---an open-and-closed subset of $X$---and to each open-and-closed $U\subseteq X\<,$ with inclusion map $i,$ the unique idempotent
$\OX\<$-ideal isomorphic to $i_*\mathcal O_U,$ that is, the ideal whose restriction
to\/ $U$ is\/ $\mathcal O_U$ and to $X\setminus U$ is $(0)$.\qed
\end{corollary}

\begin{corollary}\label{idem and Hom}
A finite-type\/ $\OX$-ideal\/ $I$ is idempotent if and only if for each $G\in\D(X)$ there exist $\D(X)$-isomorphisms, functorial in $G,$
$$
\RH_{\<X}(I,G)\simeq I\dtensor{\<\<X}G\simeq IG.
$$
\end{corollary} 

\begin{proof} If $I$ is idempotent then over the open set $U\!:=\Supp_{\<X}I$ one has
$I=\mathcal O_U$, and over the disjoint open set $X\setminus U$, $I\simeq 0$, 
so the asserted isomorphisms obviously exist  over $X=U\sqcup (X\setminus U)$.

Conversely, if   these isomorphisms hold for all members of the natural triangle
$$
I\to\OX\to \OX/I \xra{+}
$$
then, since $I(\OX/I)=0$, application of the functor $\RH(I,-)$ yields that the natural
map is an isomorphism $I\simeq I^2$ in $\D(X)$, hence in $\OX$, i.e., $I=I^2$. 
\end{proof}

\begin{corollary}\label{derived idem}
Let\/ $X$ be a locally noetherian scheme.
For a complex\/ $L\in\D(X)$ the following conditions are equivalent. 

{\rm(i)} $L$ is isomorphic in\/ $\D(X)$ to an idempotent\/ $\OX$-ideal.

{\rm(ii)} $L\in\dcatc X$ and there exists a\/ $\D(X)$-isomorphism 
$L\dtensor{\<\<X}L\iso L$.
\end{corollary}

\begin{proof} If (i) holds then $L\in\dcatc X$ is clear; and taking $G=I$ 
in~\ref{idem and Hom}, one gets (ii). 

When (ii) holds,  (i) follows easily from \cite[4.9]{AIL}.
\end{proof}

\begin{proposition}\label{g*idem}
Let\/ $g\col Z\to X$ be a morphism of local ringed spaces
$($so that for each\/ $z\in Z$ the associated stalk homomorphism
$\mathcal O_{\<\<X\<\<,\>gz}\to\mathcal O_{\<Z,\>z}$ is a local homomorphism
of local rings$)$.  Let\/ $I$ be an\/ $\OX$-ideal.
If\/ $I$ is idempotent then so is $I\mathcal O_{\<Z}\cong g^*\<I\simeq\bL g^*\<I$.  The converse  holds if\/ $g$ is flat and surjective. 
\end{proposition}

\begin{proof} If $I=I^2$ then $I\mathcal O_{\<Z}=(I\mathcal O_{\<Z})^2$. Flatness of $\OX/I$ implies that
$I$ is flat and that the natural map $g^*\<I \to g^*\OX=\mathcal O_{\<Z}$ is injective, and thus
$\bL g^*\<I\simeq  g^*\<I\cong I\mathcal O_{\<Z}$.

If $g$ is flat and surjective then for each $x\in X$ there is a $z\in Z$ such that $g(z)=x$, and then there is a flat local homomorphism\vspace{.5pt} 
$\mathcal O_{\<\<X\<\<,\>x}\to\mathcal O_{\<Z,\>z}$. 
Hence if $I\mathcal O_{\<Z}=(I\mathcal O_{\<Z})^2$ then 
$I_xO_{\<Z,\>z}=I_x^2O_{\<Z,\>z}$,
i.e., $I_x=I_x^2$. As this holds for all $x$, therefore $I=I^2$.
\end{proof}

\begin{corollary}\label{lift mult idem}
Let\/ $g\col Z\to X$ be a morphism of local ringed spaces, and $I$ an idempotent $\OX$-ideal.
  \begin{enumerate}[\quad\rm(1)]
    \item
For any\/ $E\in\D(X),$ there is a unique isomorphism 
$\bL g^*\<(IE)\simeq I\>\bL g^*\<\<E$ whose composition with the natural
map $I\>\bL g^*\<\<E\to \bL g^*\<\<E$ is the map obtained by applying\/
$\bL g^*$ to the natural map $IE\to E$.
    \item
If\/ $g$ is a perfect scheme-map then\vspace{-1.5pt} for any\/ $E\in\Dqcpl(X),$
there exists a unique isomorphism 
$g^!(IE)\simeq I\>g^!\<\<E$ whose\vspace{.7pt} composition with the natural
map \mbox{$I\>g^!\<\<E\to g^!\<\<E$} is the map obtained by applying~
$g^!$ to the natural map\/ $IE\to E$.
  \end{enumerate}
\end{corollary}

\begin{proof}
Uniqueness holds because, $I\mathcal O_{\<Z}$ being idempotent, $I\bL g^*\<\<E\simeq I\mathcal O_{\<Z}\otimes_Z \bL g^*\<\<E$ is a direct summand
of $\mathcal O_{\<Z}\otimes_Z \bL g^*\<\<E\simeq \bL g^*\<\<E$ (Proposition~\ref{char idem}, (iv) and (i$'$)).

Since both $I$ and $\OX/I$ are flat over $\OX$, there are for all $F\in \D(X)$
natural isomorphisms
$
\bL g^*\<I\dtensor{\<Z} F\simeq  g^*\<I\otimes_Z F\cong IF.
$
So for all $E\in\D(X)$,
$$
\bL g^*\<(IE)\simeq\bL g^*\<(I\dtensor{\<\<X}E)\simeq 
\bL g^*\<I\dtensor{\<Z} \bL g^*\<\<E\simeq I\>\bL g^*\<\<E.
$$
The composition of these isomorphisms has the property asserted in (1).

Similarly, if $g$ is a perfect scheme-map then, using Theorem~\ref{char perfect}, one gets natural isomorphisms for all $E\in\Dqcpl(X)$,
$$
\qquad g^!(IE)\simeq g^!(I\dtensor{\<\<X}E)\simeq 
\bL g^*\<I\dtensor{\<Z} \bL g^*\<\<E\dtensor{\<Z}g^!\OX\simeq
\bL g^*\<I\dtensor{\<Z} g^!E\simeq I\>g^!E,\qquad
$$
that compose to the isomorphism needed for (2).
\end{proof}

The next result is to the effect that \emph{idempotence satisfies faithfully flat descent} (without any ``cocycle condition").

\begin{proposition}
\label{idem gluing} 
Let\/ $g\col Z\to X$ be a faithfully flat map, and let $\pi_1\col
Z\times_X Z\to Z$ and $\pi_2\col Z\times_X Z\to Z$ be the canonical
projections. If\/ $J$ is an idempotent $\mathcal O_{\<Z}$-ideal such
that there exists an isomorphism $\pi_1^*J\cong\pi_2^*J$ then there is
a unique idempotent\/ $\OX$-ideal such that\/ $J=I\mathcal O_{\<Z}$.
\end{proposition}

\begin{proof} 
(Uniqueness.) If $J=I\mathcal O_{\<Z}=I'\mathcal O_{\<Z}$ where $I$ and
$I'$ are idempotent $\OX$-ideals with respective supports $U$ and $U'\<$,
then $g^{-1}U=g^{-1}U'\<$ (both being the support of~$J$), and since $g$
is surjective, therefore $U=U'\<$, so $I=I'$.

(Existence.) Let $V$ be the support of $J$. The support of $\pi_1^*J$ is
$\pi_1^{-1}V=V\times_X Z$, and similarly that of $\pi_1^*J$ is $Z\times_X V$.
Hence, since $\pi_1^*J\cong\pi_2^*J$, the following subsets of $Z\times_X Z$ are all the same:
$$
V\times_X Z = Z\times_X V =(V\times_X Z) \cap (Z\times_X V)= V\times_X V.
$$

If $v\in V$ and $w\in Z$ are such that $g(v)=g(w)$, then there is a
field $K$ and a map $\gamma\col\Spec K\to V\times_X Z=V\times_X V$
such that the set-theoretic images of $\pi_1\gamma$ and~$\pi_2\gamma$
are $v$ and~$w$ respectively,  so $w\in V$. Thus $V=g^{-1}g(V)$.

We claim that  $g(V)$ \emph{is open and closed in} $X$.  For this it
suffices to show that for each connected component $X'\subseteq X$,
$g(V\cap \>g^{-1}X')=X'$.  Without loss of generality,  then, we may
assume that $X$ is connected, so $X'=X$.

Since $g$ is flat, if $y\in g(V)$ then the generic point $x_1$ of any
irreducible component~ $X_1$ of $X$ containing $y$ is also in $g(V)$. In
fact $X_1\subseteq g(V)$, else the preceding argument applied to $\bar
V\!:=Z\setminus V$ would show that $x_1\in g(\bar V)=X\setminus g(V)$.
It results that some open neighborhood of $y$ is in $g(V)$; and thus
$g(V)$ is open. Similarly, $g(\bar V)=X\setminus g(V)$ is open, so $g(V)$
is closed.

The conclusion follows, with $I$ the idempotent $\OX$-ideal corresponding
to the open-and-closed set $g(V)\subseteq X$.  \end{proof}

\end{document}